%% file: Main_SE_RHS_2026.tex
\documentclass[12pt,a4paper,reqno]{amsart}
\numberwithin{equation}{section}
\usepackage[utf8]{inputenc}
\usepackage[english]{babel}
\usepackage{hyperref}

\usepackage{multirow} 
\usepackage{graphicx}
\usepackage{array}
\usepackage{longtable}
\usepackage{tikz-cd}
\newtheorem{theo}{Theorem}[section]
\newtheorem*{theorem*}{Theorem}
\newtheorem{lem}{Lemma}[section]
\newtheorem{defi}{Definition}[section]

\newtheorem{rem}{Remark}[section]
\newtheorem{prop}{Proposition}[section]
\newtheorem{exm}{Example}[section]

\usepackage[foot]{amsaddr}
\usepackage{url}
\usepackage{hyperref}

\usepackage{amsmath}
\usepackage{amssymb, amscd}
\usepackage[all, cmtip]{xy}
\usepackage{xcolor}
\usepackage{multirow} 
\usepackage{longtable}
\usepackage{array}

\newtheorem{theorem}{Theorem}[section]

\newtheorem{conjecture}[theorem]{Conjecture}

\newtheorem{proposition}[theorem]{Proposition}
\newtheorem{corollary}[theorem]{Corollary}

\theoremstyle{definition}

{\sc}% Thm head font  (could be also \sc)
{}%         Body font
{}%         Indent amount (empty = no indent, \parindent = para indent)
{\sc}% Thm head font  (could be also \sc)
\newtheorem{remark}{Remark}[section]

\newcommand{\thismonth}{\ifcase\month\or
  January\or February\or March\or April\or May\or June\or
  July\or August\or September\or October\or November\or December\fi
  \space\number\year}

\makeatletter
\newcommand{\rssymb}[2]{\newcommand{#1}{{\mathrmsl{#2}}}}
\newcommand{\calsymb}[2]{\newcommand{#1}{{\mathcal{#2}}}}
\newcommand{\bbsymb}[2]{\newcommand{#1}{{\mathbb{#2}}}}
\newcommand{\lieoper}[2]{\newcommand{#1}{\mathop
  {\mathfrak{#2}\null}\nolimits}}
\newcommand{\oper}[3][n]{\newcommand{#2}{\mathop
  {\mathrm{#3}\null}\ifx n#1\nolimits\else\limits\fi}}
\newcommand{\rsoper}[3][n]{\newcommand{#2}{\mathop
  {\mathrmsl{#3}\null}\ifx n#1\nolimits\else\limits\fi}}
\bbsymb\C{C} \bbsymb\F{F} \bbsymb\HQ{H}\bbsymb\N{N} \bbsymb\Q{Q}
\bbsymb\R{R} \bbsymb\U{U} \bbsymb\V{V} \bbsymb\W{W} \bbsymb\Z{Z}
\bbsymb\bbf{F} \bbsymb\bbk{K} \bbsymb\bbi{I} \bbsymb\bbl{L}
\bbsymb\bbo{O} \bbsymb\bbj{J} \bbsymb\bby{Y} \bbsymb\bbp{P}
\bbsymb\bba{A}
\calsymb\cA{A} \calsymb\cB{B} \calsymb\cC{C} 
\calsymb\cM{M} \calsymb\cN{N} \calsymb\cO{O} \calsymb\cP{P}
\calsymb\cU{U} \calsymb\cV{V} \calsymb\cW{W} \calsymb\cX{X}
\calsymb\cY{Y} \calsymb\cZ{Z}
\renewcommand{\geq}{\geqslant} \renewcommand{\leq}{\leqslant}
\oper\End{End} \oper\Hom{Hom}                    % Vector space constructions
\oper\Sym{Sym} \oper\Skew{Skew}
\oper\Aut{Aut}                                   % Group constructions
\oper\GL{GL} \oper\SL{SL}\oper\Symp{Sp} \oper\CO{CO} \oper\On{O}
\oper\SO{SO} \oper\Pin{Pin} \oper\Spin{Spin} \oper\CU{CU}
\oper\Un{U} \oper\SU{SU} \oper\PSU{PSU} \rsoper\Diff{Diff}
\rsoper\SDiff{SDiff}
\lieoper\der{der}                                % Lie algebra constructions
\lieoper\gl{gl} \lieoper\sgl{sl}\lieoper\symp{sp} \lieoper\co{co}
\lieoper\so{so} \lieoper\spin{spin} \lieoper\cu{cu} \lieoper\un{u}
\lieoper\su{su} \rsoper\Vect{Vect} \rsoper\Ham{Ham}
\def\la#1{\hbox to #1pc{\leftarrowfill}}
\def\ra#1{\hbox to #1pc{\rightarrowfill}}

\newcommand{\Norm}[2][]{\bigl|\mkern-3mu\bigr|#2\bigr|\mkern-3mu\bigr|
  _{\lower1pt\hbox{${}_{#1}$}}}

\rsoper\dimn{dim}                           % dimension
\rsoper\grad{grad}                          % gradient
\rsoper\kernel{ker}\rsoper\image{im}        % kernel and image
\rsoper\alt{alt}   \rsoper\sym{sym}         % alternating and symmetric part
\rsoper\Ad{Ad}     \rsoper\ad{ad}           % adjoint action or bundle
\rsoper\CoAd{CoAd} \rsoper\coad{coad}       % coadjoint action
\rsoper\trace{tr}  \rsoper\trfree{tf}       % trace and tracefree part
\rsoper\detm{det}                           % determinant
\rsoper\Vol{Vol}                            % volume
\rsoper\divg{div}                           % divergence
\rsoper\sign{sign}                          % sign function
\rssymb\iden{id}                            % identity
\rssymb\vol{vol}                            % volume element
\oper\Imag{Im}\oper\Real{Re}                % real and imaginary
\newcommand{\sd}{{\raise1pt\hbox{$\scriptscriptstyle +$}}}
\newcommand{\asd}{{\raise1pt\hbox{$\scriptscriptstyle -$}}}
\newcommand{\sdasd}{{\raise1pt\hbox{$\scriptscriptstyle\pm$}}}
\newcommand{\asdsd}{{\raise1pt\hbox{$\scriptscriptstyle\mp$}}}

\rsoper\scal{scal}
\def\kahl/{k\"ahler}
\def\Kahl/{K{\"a}hler}

\begin{document}

\title[Sasaki-Einstein $\mathbb Q$-homology spheres via rational varieties and Berglund-Hübsch]
{Sasaki-Einstein rational homology spheres via  rational varieties and the Berglund-Hübsch rule}

\author[J. Cuadros]{Jaime Cuadros Valle$^1$}
\author[J. Lope]{Joe Lope Vicente$^1$}

\address{$^1$Departamento de Ciencias, Secci\'on Matem\'aticas,
Pontificia Universidad Cat\'olica del Per\'u,
Apartado 1761, Lima 100, Per\'u}
\email{jcuadros@pucp.edu.pe}
\email{j.lope@pucp.edu.pe}

\date{\thismonth}

\begin{abstract} 
 We find  Sasaki-Einstein metrics on rational homology  $(4n-1)$-spheres for $n>1$ built from cyclic polynomials of index 1 cutting out rational varieties. The Einstein metrics found here are inequivalent to the ones discovered by Boyer and Galicki in  \cite{BG02}. 
Our findings  are a  consequence of an improvement, for hypersurfaces defined by  cyclic polynomials, on the estimate given by Johnson and Kollár  in \cite{JK} to determine Kähler-Einstein orbi\-fold metrics.  We also construct weighted hypersurfaces that contain the rational varieties described above as codimension two subvarieties and, due to the refined estimate for cyclic polynomials, we  find conditions on the weights and degrees of these hypersurfaces such that their corresponding  smooth links admit Sasaki-Einstein metrics as well. Finally, we  study the effect of the Berglund-Hübsch transpose rule, coming from classical mirror symmetry, on the topology  and on the existence of  Sasaki-Einstein metrics on the links studied. 
We generalize the results presented in \cite{CGL} for rational homology 7-spheres to rational homology $(4n-1)$-spheres by proving  that, under mild conditions, these two features are invariant under the aforementioned transpose rule. 
 \end{abstract}

\maketitle
% !TEX root =  % !TEX root =  

\medskip

\medskip

\maketitle
\vspace{-2mm} 

\input{Introduction}
\input{Chapter3}
\input{Chapter4}

\section*{Declarations}

%\subsection*{Ethics approval and consent to participate} Informed consent was obtained from all individual participants included in the study.

%\subsection*{Consent for publication}  The authors have read and understood the publishing policy, and submit this manuscript in accordance with this policy.

%\subsection*{Availability of data and materials} All of the material is owned by the authors and/or no permissions are required.

%\subsection*{Competing interests} We declare that the authors have no competing interests as defined by Springer, or other interests that might be perceived to influence the results and/or discussion reported in this paper.

%\subsection*{Funding} The first author received  financial support from Pontificia Universidad Católica del Perú through project VRI-DFI 2019-1-0089.

\subsection*{Data Availability} Data sharing not applicable to this article as no datasets were generated or analysed during the current study.

\subsection*{Conflict of interests} The authors have no competing interests to declare that are relevant to the content of this article.

%\subsection*{Authors' Contributions} J.C wrote  the manuscript. J.L made  contributions to the interpretation of data and wrote the programs in Matlab used to  determine the torsion of the homology groups.

%\subsection*{Acknowledgements} The authors would like to thank the anonymous reviewer for the valuable detailed suggestions which greatly improved the clarity and organization of this article. The second author would like to thank Marco Aldi for useful conversations.

%\bibliographystyle{plain} % Defines the style (e.g., plain, IEEEtran, alpha)
%\addcontentsline{toc}{chapter}{References}
%\renewcommand{\bibname}{References}
\bibliographystyle{amsalpha}
\bibliography{References} % Links to your_bib_file_name.bib (do not include the .bib extension)

\end{document}

%% file: Introduction.tex
% !TEX root = Main_SE_RHS_2026.tex

\section{Introduction}

From a differential geometric point of view, establishing the existence of  Einstein metrics on differentiable manifolds is  of the upmost importance since these metrics represent optimal geometries that prescribe canonical curvatures on manifolds. In \cite{Ko}  Kobayashi considered circle bundles 
$
S^1 \hookrightarrow M \rightarrow X
$
 over  Riemannian manifolds $(X, g)$ with a connection $1$-form $\eta.$ By choosing the standard metric on the circle fibers, one can lift the base metric $g$ to a metric $g_\eta$ on the total space $M$. A natural question posed and answered by Kobayashi was the following: if $(X, g)$ is an Einstein manifold, under what conditions is $(M, g_\eta)$ also Einstein? In general, this occurs only 
 when $(X, g)$ is Ricci flat and the connection  is flat or if $X$ is a complex manifold, $g$ is the real part of a Kähler-Einstein metric $\omega$ satisfying $\operatorname{Ricci}(\omega) > 0.$ The latter can be paraphrased by saying that  the link of a cone over a smooth projective variety $X$ carries a natural Einstein metric if and only if $X$ is Fano and it  carries a Kähler-Einstein metric. The Einstein metric obtained is of Sasaki type: a metric structure compatible with a natural contact structure on the link. Metrics that are simultaneously Einstein and Sasaki are called Sasaki-Einstein metrics.
 
 In \cite{BG01}  Boyer and Galicki  generalized the result given by Kobayashi to weighted cones and gave an algorithm, the Kobayashi-Boyer-Galicki method, to obtain Sasaki-Einstein manifolds from the existence of orbifold Fano Kähler-Einstein hypersurfaces in weighted projective complex $n$-spaces.  In generality,  to obtain Fano Kähler-Einstein  metrics one has to determine whether the variety is  K-stable \cite{Ti}, \cite{CDS}, \cite{LXZ}. In  \cite{Ti} Tian defined the  
$\alpha$-invariant to study K-stability for the smooth case. Later,  Demailly and Kollár in \cite{DK} 
adapted the $\alpha$-invariant in the context of orbifolds and showed that  
given an $n$-dimensional Fano variety  $X$ (possibly with quotient singularities) such that there is an $\epsilon>0$ satisfying 
$$
\left(X, \frac{n+\epsilon}{n+1} D\right) \text { is klt }
$$
for every $D\in |-K_X|_{\mathbb Q}.$ Then $X$ has a Kähler-Einstein (orbifold) metric.  Johnson and Kollár \cite{JK1, JK} improved this result for quasi-smooth weighted hypersurface: suppose $X \subset \mathbb{P}(\mathbf{w})=\mathbb{P}\left(w_0, \ldots, w_n\right)$ is a normal Fano variety of index $I$  such that

$$
I \operatorname{deg}(X)<\frac{n}{n-1} \min _{i, j}\left\{w_i w_j\right\}.
$$
Then there exists $\epsilon>0$ such that $\left(X,\frac{n-1+\epsilon}{n} D\right)$ is klt for any $D \in\left|-K_X\right|_{\mathbb{Q}}$ and then  $X$ admits a Kähler-Einstein metric. 

This method had successful outcomes during the first decade of this century to establish the existence of Sasaki-Einstein metrics on homotopy spheres and on rational homology spheres \cite{BGN, BGK}.  For instance, in \cite{BGN}, the authors give the first examples of Sasaki-Einstein metrics on highly connected rational homology $7$-spheres. These examples were constructed from certain elements of the list of Kähler-Einstein Fano 3-folds given in \cite{JK}. 
Subsequently, the results in \cite{BGN}  were extended in \cite{CL}. 

Of course, the estimate given by  Johnson and Kollár for the existence of Kähler-Einstein metrics has been improved and refined  for certain types of weighted hypersurfaces. In \cite{BGK}, the authors achieved a better estimate for Fano weighted hypersurfaces that come from polynomials of Brieskorn-Pham type and their perturbations. The proposed inequality allowed Boyer and Galicki  \cite{BG02} to find parameter families of Sasaki-Einstein metrics on infinitely many rational homology $(2n-1)$-spheres, for each $n\geq 3$, which are built as branched covers of Fermat Calabi-Yau hypersurfaces. 

In recent years the $\alpha$-invariant has been replaced by more powerful invariants.  For instance the $\delta$-invariant defined by Fujita and Odaka in \cite{FO}. Recently, in \cite{ST}, Sano and Tasin used this invariant in the context of weighted hypersurfaces  to establish  K-stability and  hence existence of a Kähler-Einstein structure on them. For this, a necessary condition is that at least one of the weights divides the degree of the hypersurface. As a consequence, they found new families of quasi-smooth $3$-folds and $4$-folds of index $1$ that are K-stable (see \cite{IF} and \cite{BK}) and thus with corresponding links admitting Sasaki-Einstein metrics. 

Another important tool one can use to construct new links that are rational homology spheres admitting Sasaki-Einstein metrics is the \textit{Berglund-Hübsch transpose rule}. This technique, which was introduced in the Berglund-Hüsbsch-Krawitz (BHK) mirror symmetry construction, was proposed by Berglund and Hübsch in \cite{BH} and resumed and expanded by Krawitz in \cite{Kr}. Originally the purpose of Berglund-Hübsch transpose rule was to determine  Calabi-Yau mirror pairs. The method consists of taking  a Calabi-Yau hypersurface $X_{f}$ in certain weighted projective space $\mathbb{P}(\mathbf{w})$ which is defined by an invertible polynomial, that is, a  quasi-smooth and  quasi-homogeneous polynomial $f=\sum_{j=1}^{n}\prod_{i=1}^{n}z_{i}^{a_{ij}}$ that is characterized by the fact that its matrix of exponents $A=[a_{ij}]$ is invertible over $\mathbb{Q}$. Now, if one considers its transpose matrix $A^T=[a_{ji}]$, one constructs a new  polynomial $f^{T}=\sum_{j=1}^{n}\prod_{i=1}^{n}z_{i}^{a_{ji}}$, which is also invertible (see \cite{KS}) and its corresponding Calabi-Yau hypersurface, the {\it dual} hypersurface $X_{f^{T}}$, which is contained in a different  weighted projective space $\mathbb{P}(\mathbf{w^{T}}).$ 

In  \cite{CGL}, we studied the effect of the Berglund-Hübsch rule on links that are rational homology $7$-spheres arising  from the list of Johnson and Kollár of Fano 3-folds anticanonically embedded in weighted projective spaces. In that work, we proved that the Berglund-Hübsch rule preserves important topological and Riemannian features:  
\begin{itemize}
    \item The homology groups  of the links $L_{f}$ and $L_{f^{T}}$ coincide for some particular types of invertible polynomials.
    \item The existence of Kähler-Einstein metric on the orbifold $X_{f}$ is also preserved by the Berglund-Hübsch rule under some mild conditons. In particular, if  the link $L_f$ admits a Sasaki-Einstein metric then its dual $L_{f^T}$ admits a Sasaki-Einstein metric as well.
\end{itemize}

In this work we  generalize the results given in \cite{CGL} for links of  dimension $4n-1.$ Indeed, first  we  study the effect of the Berglund-Hübsch rule on the topology of rational homology $(4n-1)$-spheres given as links defined by certain types of invertible polynomials. Then we impose conditions on the data $(\mathbf{w},d)$ of weights and degree associated to these polynomials so that the Sasaki-Einstein property is preserved under the Berglund-Hübsch rule.  Below we roughly explain how to achieve these generalizations. 
\medskip

\subsection*{Building rational homology $(4n-1)$-spheres}
The main ingredient to build the rational homology spheres  studied in this paper consists  of \textit{cycle} polynomials, that is, polynomials that can be written as:
\begin{equation}
f_{0}=z_{n}z_{0}^{a_{0}}+z_{0}z_{1}^{a_{1}}+\dots+z_{n-1}z_{n}^{a_{n}}.
\end{equation}
In \cite{K}, Kollár studied many important properties related to the orbifold hypersurface $X_{f_{0}}\subset\mathbb{P}(\mathbf{w})$ that are obtained as a consequence of the description of the degree $d$ and the weights $w_{i}$'s in terms of the exponents $a_{0},a_{1},\dots, a_{n}$ of $f_{0}$. For instance, it is not difficult to prove that  the degree of $f_0$ equals the determinant of the matrix $A$ of exponents of $f_0$ under the condition $\gcd(w_o. \ldots w_n)=1.$ Kollár considered the map
$$
\phi_A: \mathbb{P}\left(w_0, w_1, \ldots, w_n\right) \rightarrow \mathbb{P}^n
$$ 
given by 
$
\left(x_0: x_1: \ldots: x_n\right) \stackrel{\phi_A}{\longmapsto}\left(y_0: y_1: \ldots: y_n\right)
$ 
with $y_j=\prod_{i=0}^n x_i^{a_{i j}},$  and he showed that $\phi_A$ maps  $\mathbb{P}\left(w_0, w_1, \ldots w_n\right)$ birationally to $\mathbb{P}^n,$ and so $\{f_0=$ $0\}$ is birational to the hyperplane $\left\{y_0+y_1+\ldots+y_n=0\right\} \subset \mathbb{P}^n$ (see \cite{Ess} for a slight generalization of this argument). 
Moreover, when $n$ is even, one can use the Milnor-Orlik formula \cite{MO} for weighted homogeneous polynomials and 
showed that the corresponding  link $L_{f_{0}}=\{f_0=0\}\cap S^{2n+1}$ is a rational homology sphere. 

We  construct new invertible polynomials adding  blocks of the form  $g(z_{n+1},z_{n+2})$ to the cyclic polynomial $f_{0}$. In this way, we obtain Thom-Sebastiani sums of the following three forms: 
\begin{itemize}
\item $$f_{I}=f_{0}+z_{n+1}^{a_{n+1}}+z_{n+2}^{a_{n+2}},$$ 
\item $$f_{II}=f_{0}+z_{n+1}^{a_{n+1}}+z_{n+1}z_{n+2}^{a_{n+2}} \quad \text{ or } \quad $$
\item $$f_{III}=f_{0}+z_{n+2}z_{n+1}^{a_{n+1}}+z_{n+1}z_{n+2}^{a_{n+2}}.$$  
\end{itemize}

Giving suitable weights to the variables $z_{n+1}$ and $z_{n+2}$, we find that the links corresponding to these Thom-Sebastiani sums are also rational homology spheres. Moreover, for  invertible polynomials of the form $f=f_{I}$ or $f=f_{III}$, we  show that the Berglund-Hübsch transpose rule leaves the homology groups invariant, that is, the homology groups  of the links $L_{f}$ and $L_{f^{T}}$ are the same. In the remaining case, when the invertible polynomial is of the form  $f=f_{II}$, we show  that although the homology groups of the links do not coincide, they remain rational homology spheres under the Berglund-Hübsch rule.  

\subsection*{Sasaki-Einstein metrics on rational homology spheres}
In the last part of this article, we focus on establishing the existence of Sasaki-Einstein metrics on the rational homology spheres built above. 
First, we study the Fano  weighted hypersurfaces defined by cyclic polynomials $f_0$ which define rational varieties. We show that these varieties satisfy the estimate given by  Johnson and Kóllar in \cite{JK} and thus admit Kähler-Einstein orbifolds metrics provided the Fano index  1. Our proof is based  on careful calculations to obtain effective bounds to establish the estimate of  Johnson and Kollár (see Section 6). Actually, we  improve  this original estimate for links that come from cycle polynomials with the condition $\gcd(w_0, \ldots w_n)=1$ and Fano index 1 (Theorem \ref{prop:4.3.4} ): 
  $$d <\frac{n+2}{n+1} \min _{i, j}\left\{w_i w_j\right\}.$$ 
  Since $\frac{n+2}{n+1} \min _{i, j}\left\{w_i w_j\right\}<\frac{n}{n-1} \min _{i, j}\left\{w_i w_j\right\},$ we have:

\begin{theorem*}(Theorem \ref{prop:4.3.4})
 Consider the cycle polynomial   
 $$f_0=z_n z_0^{a_0}+z_0 z_1^{a_1}+z_1 z_2^{a_2}+\cdots+z_{n-1} z_n^{a_n},$$  with $n \geq 4$ even. Suppose that the weight vector $\mathbf{w}$ and degree $d$ satisfy the following: 
 \begin{enumerate}
 \item  $\gcd(w_0, \ldots w_n)=1.$
 \item The Fano index equals 1.
 \end{enumerate}
 Then the Fano weighted hypersurface $X_{f_0} \subset \mathbb{P}(\mathbf{w})$ admits an orbifold  Kähler-Einstein metric and  from  the Kobayashi-Boyer-Galicki method it follows that the corresponding  $S^1$-orbibundle admits automatically Sasaki-Einstein structures. Furthermore, all these links are rational homology $(4n-1)$-spheres.  
 \end{theorem*}

\begin{remark}
Relying on this new estimate, one can show, following a similar arguments as the ones used in Corollary 12 and 13 in \cite{JK1},  that for $X_{f_0}$ 
there is an $\epsilon>0$ such that
$$
\left(X_{f_0}, \frac{n+1+\epsilon}{n+2} D\right) \text { is klt }
$$
for every $D\in |-K_{X_{f_0}}|_{\mathbb Q}.$

\end{remark}

%Cycle polynomials have remarkable properties, for instance consider the  map $$\phi_A: \mathbb{P}(w_0,w_1,\ldots , w_n)\rightarrow  \mathbb{P}^{n}$$ 
%given by 
%$$\left(x_0: x_1: \ldots : x_n\right) \stackrel{\phi_A}{\longmapsto}\left(y_0: y_1:\ldots :y_n\right)$$ with  $\quad y_j=\prod_{i=0}^n x_i^{a_{i j}},$  which maps 
% $\mathbb{P}(w_0,w_1,\ldots  w_n)$ birationally to  $ \mathbb{P}^{n}.$ Kollár showed in \cite{Kol} that if 
% 
% It follows that $\{ f=0\}$ is mapped birationally to the  hyperplane 
%$\left \{ y_0+y_1+\ldots +y_n=0\right \}\subset \mathbb{P}^{n}$ Thus, the cycle polynomial described above  cuts out a weighted hypersurface  which is  birational to $\mathbb P^{n-1}.$

Under the arithmetic restrictions on the weights and index of the weighted hypersurface  $X_{f_0}$, one has  
the  birational map $\phi_A:  \mathbb{P}({\bf w}) \dashrightarrow \mathbb{P}^{n}$  described as above, whose restriction 
$\phi_A|_{_{X_{f_0}}}: X_{f_0} \dashrightarrow \mathbb{P}^{n-1}$ is still birational and induces a correspondence   
$\Phi_A:L_{f_0} \dashrightarrow S^{2n-1}$  on the associated $S^1$-Seifert bundles which is transversely birational. Pictorially, we have the diagram: 
  $$
 \xymatrixcolsep{5pc}
\xymatrix{
 & \mathbb{P}({\bf w}) &  \mathbb{P}^n \ar@{<--}[l]_{\quad \phi_A} \\
 & X_{f_0} \ar@{^{(}->}[u]  &  \ar@{<--}[l]_{\phi_A}\mathbb{P}^{n-1} \ar@{^{(}->}[u] \\
L_f \ar[ru]^{\pi}  & \ar@{<--}[l]_{\Phi_A} S^{2n-1}\ar[ru]^{\pi} & 
}
$$
Here $\pi$ denotes the orbifold Riemannian submersions of the $S^1$-Seifert bundles. In particular, the map $\pi: S^{2 n-1} \rightarrow \mathbb{P}^{n-1}$ is the Hopf fibration which pulls back the Fubini-Study metric of $\mathbb{P}^{n-1}$ to a Sasaki-Einstein metric on the sphere $S^{2 n-1}$ up to rescaling. Thus, the map $\Phi_A$  is  a transversely birational map that  preserves the homology groups at rational level and  the property of being Sasaki-Einstein. Thus, from this diagram, one  can extract  a complex  {\it rational} Hopf fibration $L_f\xrightarrow{\pi} X_{f_0}$ on which we successfully construct Sasaki-Einstein metrics. 
\medskip

On the other hand, due to our improved estimates for the cyclic case, we are able to find conditions on the weights and degree of the Thom-Sebastiani sums of type  $f=f_{I}$ or $f=f_{III}$, so that they admit Kähler-Einstein metrics,  in case the Fano index equals 1 or 2, so  $L_{f}$ admits Sasaki-Einstein metrics for these two cases. We also show that  the Berglund-Hübsch duals for these Sasaki-Einstein links  also admit Sasaki-Einstein structure for these two type of polynomials. Lastly, we study the case $f=f_{II}$ where the Fano index equals $1$ or $2$. For this, it suffices to give some numerical condition on the weights to ensure that dual link  $L_{f^{T}}$ carries a Sasaki-Einstein structure, without assuming existence of Sasaki-Einstein on the original link $L_f.$ However,  this time the weights obtained via the Berglund-Hübsch  are such  the dual polynomial $f^T$ is not well-formed. In the latter case, it is of interest to show that the original link admits Kähler-Einstein metric, however the estimates based on  $\alpha$-invariant used in this paper are  insufficient to address this problem.  The rather recent theory of admissible flags introduced by  Abban an Zhuang in \cite{AZ} along with  its refinements (see for instance  \cite{ACC+}) to compute global $\delta$-invariants  can be useful to  determine whether this is the case, at least on rational homology 7-spheres. We plan to explore this in a future project. 
\medskip

Besides its importance in differential geometry and the strong connections 
 with complex algebraic geometry, it is well known that Sasaki-Einstein metrics  have real Killing
spinors \cite{FK} which play an important role in the context of supergravity and string theory. 
Moreover, the existence of these special metrics  reinforces the validity of the AdS/CFT correspondence from string theory in theoretical physics, a correspondence that establishes connections between quantum field theory and quantum gravity (\cite{GMSY},\cite{XY}, see also \cite{GMS}).

The paper is organized as follows. In Sections 2 and  3 we briefly review the  preliminary material needed to prove our results.  In Sections 4 and 5, using the algorithms given by Milnor and Orlik we obtain the numerical conditions to obtain links with the rational homology of a sphere,  and then study the effect of the Berglund-Hübsch transpose rule on the topology of these links . Finally, in Section 6 and 7, we explain our main results on the existence of Sasaki-Einstein metrics on rational homology spheres and their invariance under the Berglund-Hübsch transpose rule. 

%% file: Chapter3.tex
% !TEX root = Main_SE_RHS_2026.tex

%\section{Sasaki-Einstein metrics and link of hypersurface singularities}
%In  this chapter we study the natural Sasakian structures on links of isolated hypersurface singularities. In the first part, we endow a Sasakian-Einstein structure on links through the deformation of a Sasakian structure \cite{Ta}. Then the Sasakian-Einstein structure on the link is obtained as a consequence of the existence of Kähler-Einstein orbifold metrics on the associated hypersurface (see \cite{BG01}). In the second section, we explain how to determine the topology of the links through well-known formulas deduced in \cite{Mil} and \cite{MO}.

\section{Sasakian-Einstein structure on links}%
A  $(2n+1)$-dimensional compact Riemannian manifold ($M, g$) is said to be a Sasakian manifold if there is a complex structure $J$ on the metric cone  $$(C(M),\bar{g})=(M \times \mathbb{R}^{+}, \bar{g}=d r^2+r^2 g)$$ for which the metric $\bar{g}$ is Kähler. The complex structure can be extended  to the affine cone 
 $\widehat{C}(M):=C(M) \cup 0$ and one obtains a  Kähler form  $\omega_{\widehat{C}(M)}:=\frac{\sqrt{-1}}{2} \partial \bar{\partial} r^2$ on the affine cone.  
 From this definition, one obtains three important tensors $(\xi, \eta, \Phi)$ associated to the Sasakian structure, see \cite{BG, Sp} for a complete description of these structural tensors. For instance the first tensor $\xi$ is obtained as follows: consider 
 the vector field   $$\hat{\xi}:=-J\left(r \frac{\partial}{\partial r}\right)$$  on the cone  defined. It is not difficult to show that $\hat{\xi}$ is real holomorphic Killing vector field and with norm $\bar{g}(\hat{\xi}, \hat{\xi})=r^2.$ Its restriction $\hat{\xi}|_{r=1},$ denoted by $\xi$, is  {\it Reeb vector field} on $M.$   
The closure of the flow generated by the Reeb vector field $\xi$ is a  $k(\xi)$-dimensional compact torus. We denote its complexification by $\left(\mathbb{C}^*\right)^{k(\xi)}$. If $k(\xi)=1$ we say that $(M, g)$  is quasi-regular and, if in  addition,  the action induced by $\mathbb{C}^*$ on $C(M)$ is free, we say that $(M, g)$ is regular.  The pair  $(M, g)$ is said to be irregular   if $k(\xi)>1.$ 

\subsubsection*{Sasaki-Einstein structures} 
The proposition below is a well-known result in Sasakian geometry \cite{BG,Sp}:
\begin{proposition} The  following are equivalent:
\begin{itemize}
\item  $(C(M), \bar{g})$ is a Ricci-flat Kähler cone metric.
\item ($M, g$) satisfies $\operatorname{Ric}_g=2n g.$ 
\end{itemize}
Moreover, if the Sasakian structure is quasi-regular, then  $(M, g)$  is a Sasaki-Einstein manifold if and only if the quotient $$C(M) /\mathbb{C}^*=M/S^1$$ admits  a natural branch divisor $\Delta$ so that $\left(C(M) / \mathbb{C}^*, \Delta\right)$ is a log  $\mathbb{Q}$-Fano variety with  Kähler-Einstein orbifold metric.
\end{proposition}
 The last statement, sometimes referred as the {\it Kobayashi-Boyer-Galicki correspondence}, is a generalization of a result of Kobayashi \cite{K} for smooth manifolds, given in \cite{BG01}. In this article, we mostly work with orbifolds with no branch divisors, that is, $\Delta$=0.

\subsubsection*{Links of  weighted homogenous  isolated singularities}

From the Kobayashi-Boyer-Galicki correspondence,  one can obtain  Sasaki-Einstein metrics from orbifolds embedded in weighted projective spaces. 
 Recall  weighted projective space $$\mathbb {P}(w_0, w_1 \ldots w_n)=(\mathbb{C}^{n+1}-\{{\bf 0}\})/\mathbb C^*,$$  where the $\mathbb C^*$-action is the one induced by 
$\xi_{\bf w}$  and can be given by the equivalence relation 
$(z_0, \ldots, z_n) \sim\ (\lambda^{w_0}z_0, \ldots, \lambda^{w_n}z_n)$ for  $\lambda \in \mathbb{C}-\{0\}.$     We say that the weighted projective space $\mathbb{P}(w_{0},\dots,w_{n})$ is \textit{well-formed} if $\gcd(w_{0},\dots,\hat{w}_{i},\dots,w_{n})=1$, for each $i$.

A polynomial $f \in \mathbb{C}\left[z_0, \ldots, z_n\right]$ is said to be a weighted homogeneous polynomial of degree $d$ and weight vector $\mathbf{w}=$ $\left(w_0, \ldots, w_n\right)$,  if for any $\lambda \in \mathbb{C}^{*}$
$$
f\left(\lambda^{w_0} z_0, \ldots, \lambda^{w_n} z_n\right)=\lambda^d f\left(z_0, \ldots, z_n\right).$$
From the affine algebraic variety  $V_f=\{f=0\} \subset \mathbb{C}^{n+1}$ one constructs the  weighted hypersurface 
$$Z_f=(V_f-\{\mathbf{0}\})/ \mathbb{C}^* \subset Y=\mathbb{P}(\mathbf{w}).$$ 
If the cone $V_f$ is smooth everywhere except at the origin in $\mathbb{C}^{n+1}$ one says that  $Z_f$  is quasismooth. Notice that  quasismooth weighted hypersurfaces  have  only cyclic singularities. The weighted hypersurface is well-formed if $\mathbb{P}\left(w_0, \ldots, w_{n}\right)$ is well-formed and  the intersection of  $Z_f$ with the singular set of $\mathbb{P}\left(w_0, \ldots, w_{n}\right)$  has codimension at least 2 in $Z_f$. Recall that the Fano index is defined as the largest integer such that $\frac{c_1(-K_X)}{I}$ is an integer class (in the orbifold sense). In case our orbifold is given by 
a  weighted hypersurface  degree $d$, the index  is given by  $I=|\mathbf{w}|-d$ where $|\mathbf{w}|$ denotes the sum of the weights, see \cite{BG}.  When $Z_f$ is well-formed, the canonical divisor satisfies the adjunction formula $K_{Z_f}=\mathcal{O}_{Z_{f}}\left(d-|\mathbf{w}|\right).$

 For well-formedness of the weighted variety we have the following criterion \cite{IF}: 
 \begin{lem}\label{lem:2.1}
 A hypersurface defined by the weighted homogeneous polynomial $f$ of degree $d$ is well-formed in the well-formed weighted projective space 
 $\mathbb{P}\left(w_0, \ldots, w_n\right)$ if 
 $$\operatorname{gcd}\left(w_0, \ldots, \hat{w}_i, \ldots, \hat{w}_j, \ldots, w_n\right) \mid d$$
for distinct $i, j=0, \ldots, n$.  
\end{lem}
\medskip
 
 For  quasismooth weighted homogeneous polynomial $f$, the link 
$$
L_f=V_f \cap S^{2n+1},
$$
with  $S^{2n+1}$  a sufficiently small $(2n+1)$-sphere centered at the origin,  is a closed $(n-2)$-connected $(2n-1)$-manifold that bounds a parallelizable manifold with the homotopy type of a bouquet of $n$-spheres \cite{Mil}. The weighted vector field $\xi:=\sum_{i=0}^n w_i z_i \partial_{z_i}$ is the  Reeb vector field on $V_f$  which  induces the  weighted $\mathbb{C}^*$-action on $V_f$ that restricts to a weighted $S^1$-action on $S^{2n+1}$. It is well-known  that  $L_f$ admits a quasi-regular Sasaki structure by restricting the weighted Sasakian structure of the sphere to the link. Moreover, the quotient space of the link $L_f$ by the weighted  $S^1$-action is  the weighted hypersurface $Z_f$, a Kähler orbifold embedded in some weighted projective space 
${\mathbb P}({\bf w})$.  We have the following commutative diagram \cite{BG}:

\begin{equation*}
\begin{CD} 
 L_{f} @> {\qquad\qquad}>>  S^{2n+1}_{\bf w}\\
@VV{\pi}V  @VVV\\
X_{f}  @> {\qquad\qquad}>>  {\mathbb P}({\bf w}),
\end{CD}
\end{equation*}
where $S_{\mathbf{w}}^{2 n+1}$ denotes the  $(2n+1)$-sphere with a weighted Sasakian structure. The top horizontal arrow is a Sasakian embedding and the bottom arrow is a Kählerian embedding and  the vertical arrows are orbifold Riemannian submersions. 
From the Kobayashi-Boyer-Galicki correspondence, finding Sasaki-Einstein metrics on the link $L_f$ boils down to  finding Kähler-Einstein orbifold metric on $X_f.$ 
 
In \cite{Ti} Tian showed that a  Fano manifold of dimension $n$  with $\alpha$-invariant  greater
than $\frac{n}{n+1}$ admits a Kähler–Einstein metric. A generalization of this result for Fano varieties with
quotient singularities was given by Demailly and Koll\'ar (see \cite{DK}, Criterion 6.4). For the particular case of a Fano variety given as a  quasismooth weighted hypersurface,  this criterion gives  the  following estimate, whose proof can be found in \cite{JK}. 

\begin{theo}\label{prop:1.5.8}
Let $Z_f \subset \mathbb{P}\left(w_0, \ldots w_n\right)$ be a quasismooth weighted homogeneous Fano hypersurface of degree $d$. Then $Z_f$ admits a Kähler-Einstein orbifold metric if the following estimate holds:
\begin{equation}\label{eq:3.1}
d I<\frac{n}{n-1} \min _{i, j}\left\{w_i w_j\right\}. 
\end{equation}
Equivalently, the corresponding  link $L_f$ admits Sasaki-Einstein metric. 
\end{theo}

%\begin{remark} The criterion above  is useful when all weights
%are relatively big, however, there are restrictions to the size of the Fano index. This is due to an obstruction to the existence of Sasaki-Einstein (and therefore to the existence of Kähler-Einstein metrics on the corresponding Fano orbifold) discovered in \cite{GMSY}, known as the Lichnerowicz obstruction which for 
%weighted hypersurfaces of degree $d$ in $\mathbb{P}\left(w_0, \ldots w_n\right)$ reads as follows: if 
%$$
%I=|\mathbf{w}|-d>n \min _i w_i,
%$$
%then the link $L_f$ does not admit a Sasaki-Einstein metric.
%\end{remark}
%

As an important  subset of quasis-smooth weighted homogeneous polynomials  we can consider invertible polynomials.  
%An invertible polynomial  is of the form 
%$$f=\sum_{i=1}^n \prod_{j=1}^n x_j^{a_{i j}},$$ where the matrix of exponents $A=\left(a_{i j}\right)_{i, j=1}^n$ is a non-negative integer-valued matrix which is invertible over $\mathbb{Q}$ and where $f$ is weighted homogeneous and  quasismooth. 
Due to the Kreuzer-Skarke classification of invertible polynomials \cite{KS} we know that any invertible polynomial, up to permutation of variables, can be written as a Thom-Sebastiani sum of three types of polynomials usually called atoms:
\begin{enumerate}
\item  Fermat (or Brieskorn-Pham) type: $w=x^a$,
\item Chain type: $w=x_0^{a_0} x_1+x_1^{a_1} x_2+\ldots+x_{n-1}^{a_{n-1}} x_n+x_n^{a_n}$, and
\item Loop or cycle type: $w=x_0^{a_0} x_1+x_1^{a_1} x_2+\ldots+x_{n-1}^{a_{n-1}} x_n+x_n^{a_n} x_0$.
\end{enumerate}
The polynomials that can be expressed as the sum of atomic polynomials $f$ and $g$:
$$f+g=f(z_{0},\dots, z_{j})+g(z_{j+1},\dots, z_{n})$$
are called \textit{Thom-Sebastiani sums} which are again invertible.

\medskip

\section{Topology of links}
 The fact that the link $L_{f}$ is $(n-2)$-connected implies that the only homology groups $H_{k}(L_{f},\mathbb{Z})$ that can be non-trivial is when $k=0,n-1,n,2n-1$.

From  the Wang exact sequence of the Milnor fibration we have (after applying the Alexander duality and the Poincaré duality):
\begin{equation}\label{eq:3.4}
0\longrightarrow H_{n}( L_{f},\mathbb{Z})\longrightarrow H_{n}(F,\mathbb{Z})\xrightarrow{\mathbb{I}-h_{*}} H_{n}(F,\mathbb{Z})\longrightarrow H_{n-1}(L_{f},\mathbb{Z})\longrightarrow0.
\end{equation}
Here, the map $h$ represents the \textit{monodromy map}, which for weighted polynomials is  obtained from the weights $w_{j}$'s defining the  $\mathbb{C}^{*}$-action on $\mathbb{C}^{n+1}.$ 
Since (\ref{eq:3.4}) is exact, we have
$$H_{n}( L_{f},\mathbb{Z})\cong \ker(\mathbb{I}-h_{*}) \quad \text{and} \quad H_{n-1}(L_{f},\mathbb{Z})\cong \operatorname{coker}(\mathbb{I}-h_{*}).$$
Clearly, $H_{n}( L_{f},\mathbb{Z})$ is a free group, while $H_{n-1}(L_{f},\mathbb{Z})$ can admit torsion. Also notice that  $H_{n-1}(L_{f},\mathbb{Z})_{free}\cong H_{n}(L_{f},\mathbb{Z})$. Thus, we will only  compute $H_{n-1}(L_{f},\mathbb{Z})$. For this, we use the Alexander polynomial of the link.
\begin{defi}\label{def:3.2.2}
    We define the \textit{Alexander polynomial} of the link $L_{f}$ as the characteristic polynomial of $h_{*}$:
    $$  \Delta(t)=\det(t\mathbb{I}-h_{*}).$$
\end{defi}
From the exact sequence in (\ref{eq:3.4}),  one has to determine  $\Delta(1)$ to determine $H_{n-1}(L_{f},\mathbb{Z})$. 
Here, we have the following results given in \cite{BG} or \cite{D}.
\begin{lem}\label{lem:3.2.3}
    Let $\Delta(t)$ be the Alexander polynomial of the link $L_{f}$. Then
    \begin{itemize}
        \item[i)] The link $L_{f}$ is a rational homology sphere if and only if $\Delta(1)\neq0$.
        \item[ii)] The link $L_{f}$ is a homology sphere if and only if $|\Delta(1)|=1$.
        \item[iii)] If $L_{f}$ is a rational homology sphere, then the order of $H_{n-1}(L_{f},\mathbb{Z})$ is equal to $|\Delta(1)|$.
    \end{itemize}
\end{lem}
Likewise, when $\Delta(1)=0$, we see that $(t-1)$ is a factor of $\Delta(t)$. Milnor and Orlik proved that in this case, the number of factors $(t-1)$ in the polynomial $\Delta(t)$ is equal to the Betti number $b_{n-1}(L_{f})$ of the link $L_{f}$, i.e. this represents the free part of $H_{n-1}(L_{f},\mathbb{Z})$. We begin by computing this value. 

\begin{defi}\label{def:3.2.4}
    Let $f$ be a monic polynomial $f(t)=(t-\alpha_{1})(t-\alpha_{2})\cdots(t-\alpha_{k})$, where each root $\alpha_{j}\in\mathbb{C}^{*}$. We define the \textit{divisor}  of $f$ as $$\operatorname{div}f=\langle\alpha_{1}\rangle + \langle\alpha_{2}\rangle+\cdots+\langle\alpha_{k}\rangle,$$
    which is an element of the integral group ring $\mathbb{Z}[\mathbb{C}^{*}]$.
\end{defi}
In particular, for the polynomial $t^{m}-1$, we have the following notation:
\begin{equation}\label{eq:3.5}
     \Lambda_{m}:=\operatorname{div}(t^{m}-1) = \langle 1\rangle + \langle \xi\rangle+\cdots+\langle \xi^{m-1}\rangle,
\end{equation}
where $\xi=e^{\frac{2\pi i}{m}}$. Moreover, for $a,b\in\mathbb{Z}^{+}$, these verify the rule
\begin{equation}  \label{eq:3.6}  \Lambda_{a}\Lambda_{b}=\gcd(a,b)\Lambda_{\operatorname{lcm}(a,b)}.
\end{equation}
Next, we compute the divisor of the Alexander polynomial $\Delta(t)$ of $f$. This formula was proven by Milnor and Orlik in \cite{MO}.

\begin{lem}\label{lem:3.2.5}
    Let $f\in\mathbb{C}[z_{0},\dots, z_{n}]$ be a weighted homogeneous polynomial of degree $d$ and weight vector $\mathbf{w}=(w_{0},\dots,w_{n})$ with isolated singularity at the origin. The divisor of the Alexander polynomial $\Delta(t)$ of $f$ is determined by the formula
    \begin{equation}\label{eq:3.7}
        \operatorname{div} \Delta = \prod_{i=0}^{n}\left(\dfrac{\Lambda_{u_{i}}}{v_{i}}-\Lambda_{1}\right),
    \end{equation}
    where the integer numbers $u_{i}$'s and $v_{i}$'s are calculated as
    \begin{equation}\label{eq:3.8}
        u_{i}=\dfrac{d}{\gcd(d,w_{i})} \quad \text{and} \quad v_{i}=\dfrac{w_{i}}{\gcd(d,w_{i})}.
    \end{equation}
    Moreover the Milnor number $\mu\left(L_f\right)$ of the Milnor fiber 
    is given by 
    \begin{equation}\label{eq:3.6}
    \mu\left(L_f\right)=\prod_{j=0}^n\left(\frac{d}{w_j}-1\right).
    \end{equation}
    
\end{lem}

%When we compute Formula (\ref{eq:3.7}), we arrive at the following expression for the divisor of $\Delta(t)$:
%$$\operatorname{div} \Delta = \sum_{j} r_{j}\Lambda_{j},$$
%where the index $j$ runs over the set $\left\{\operatorname{lcm}(u_{i_{1}},\dots,u_{i_{s}}): \{i_{1},\dots,i_{s}\} \text{ is a subset of }I_{n}\right\}$. Here $I_{n}=\{0,1,\dots,n\}$ and we consider $\operatorname{lcm}(\emptyset)=1$. Then, the characteristic polynomial $\Delta(t)$ is given by
%\begin{equation}\label{eq:3.9}
%    \Delta(t)=\prod_{j}(t^{j}-1)^{r_{j}},
%\end{equation}
%where the $r_{j}$'s are given as above. Therefore, we have that the Betti number $b_{n-1}(L_{f})$ satisfies
%\begin{equation}\label{eq:3.10}
%    b_{n-1}(L_{f})=\sum_{j}r_{j}.
%\end{equation}
In order to compute the torsion of the group $H_{n-1}(L_{f},\mathbb{Z})$, Orlik proposed the following conjecture (see \cite{Or}):

\begin{conjecture}[\bf Orlik algorithm]
Let $f$ be a weighted homogeneous polynomial of degree $d$ and weight vector $\mathbf{w}=(w_{0},\dots, w_{n})$. We define the collection of pairs $\{(c_{i_{1},\dots,i_{s}},k_{i_{1},\dots,i_{s}})\}$, where each $\{i_{1},\dots,i_{s}\}$ is a ordered subset of $\{0,1,\dots,n\}$, i.e. $0\leq i_{1}<\dots<i_{s}\leq n$, as follows. The numbers $c_{i_{1},\dots, i_{k}}$ are defined inductively as $c_{\emptyset}=1$ and 
      $$c_{i_{1}, \ldots, i_{s}}=\dfrac{\operatorname{gcd}\left(u_{0}, \ldots, \hat{u}_{i_{1}}, \ldots, \hat{u}_{i_{s}}, \ldots, u_{n}\right)}{\prod_{J} c_{j_{1}, \ldots, j_{t}}}, 
$$
where $J$ is the collection of ordered proper subset of $\{i_{1},i_{2},\dots, i_{s}\}$ and $u_{i}$ is defined as in (\ref{eq:3.8}). On the other hand, the values $k_{i_{1},\dots,i_{s}}$ are obtained as $k_{\emptyset}=\epsilon_{n+1}$ and 
$$ k_{i_{1}, \ldots, i_{s}}=\epsilon_{n-s+1} \sum_{I}(-1)^{s-t} \frac{u_{j_{1}} \cdots u_{j_{t}}}{v_{j_{1}} \cdots v_{j_{t}} \operatorname{lcm}\left(u_{j_{1}}, \ldots, u_{j_{t}}\right)},$$
where $I$ is the collection of all ordered subsets of $\{i_{1},\dots, i_{s}\}$, $u_{i}$ and $v_{i}$ are defined as in (\ref{eq:3.8}) and 
$$
\epsilon_{n-s+1} =
\begin{cases}
0, & \text{if } n - s + 1 \text{ is even},\\
1, & \text{if } n - s + 1 \text{ is odd}.
\end{cases}
$$
We consider $r=\lfloor \max\{k_{i_{1},\dots,i_{s}\}}\}\rfloor$ and for each $j\in\mathbb{Z}$ such that $1\leq j\leq r$ we define the integer number 
$$d_{j}=\prod_{k_{i_{1},\dots,i_{s}}\geq j}c_{i_{1},\dots,i_{s}}.$$
Then the torsion of the homology group $H_{n-1}(L_{f},\mathbb{Z})$ is determined by
\begin{equation}\label{eq:3.11}
    H_{n-1}(L_{f},\mathbb{Z})_{tor}=\mathbb{Z}_{d_{1}}\oplus \cdots\oplus \mathbb{Z}_{d_{r}}.
\end{equation}
\end{conjecture}

For $n\geq3$, this conjecture was  proved by Hertling and Mase \cite{HM} for invertible polynomials.

%% file: Chapter4.tex
% !TEX root = Main_SE_RHS_2026.tex

\section{Building rational homology spheres via invertible polynomials}
In this section, we construct links that are rational homology spheres. These links are provided by polynomials of cycle type or certain types of Thom-Sebastiani sums. These types of polynomials were studied by Boyer et al. in \cite{BGN}. More precisely, they worked with the associated weights $\mathbf{w}=(w_{0},w_{1},w_{2},w_{3},w_{4})$ to these polynomials $f$ where the link $L_{f}$ has dimension $7$ and the corresponding weighted hypersurface $X_{f}$ has index $I=|\mathbf{w}|-d=1$. Here, we show that this also works for links of dimension $4n-1$ and any index $I>0$. 

\subsection{Rational homology spheres from cyclic polynomials} Consider the polynomial of cycle type
\begin{equation}    \label{eq:4.1}f_{0}=z_{n}z_{0}^{a_{0}}+z_{0}z_{1}^{a_{1}}+z_{1}z_{2}^{a_{2}}+\cdots+z_{n-1}z_{n}^{a_{n}}, \quad n\geq2 \text{ even}.
\end{equation}
In order to find the weight vector $\mathbf{v}$ and its degree $m_{\alpha}$, we must solve the following system of linear equations in variables $v_{i}'$ and $d'$:
\begin{equation}\label{eq:4.2}
    v_{n}'+a_{0}v_{0}'=d', \ \ v_{0}'+a_{1}v_{1}'=d', \quad  \cdots \quad , v_{n-1}'+a_{n}v_{n}'=d'.
\end{equation}
Following \cite{K}, we observe that the  solutions for the system (\ref{eq:4.2}) are given by the vectors $\lambda(v_{0}',v_{1}',\dots, v_{n}',d')$, with $\lambda\in\mathbb{R}$ and
\begin{equation}\label{eq:4.3}
    v_{i}'=1-a_{i+1}+a_{i+1}a_{i+2}-\dots +a_{i+1}a_{i+2}\dots a_{i+n-1}a_{i+n},
\end{equation}
where the subscripts are taken$\mod (n+1)$ and $d'=1+a_{0}a_{1}\dots a_{n}$. Here, it is possible that the values $v_{i}'$ have a common factor. Then, we write $v^*=\gcd(v_{0}',v_{1}',\dots,v'_{n})$. Thus, we can take the weight vector and the degree of $f_{0}$ as $\textbf{v}=\frac{1}{v^*}(v_{0}',v_{1}',\dots,v_{n}')$ and $m_{\alpha}=\frac{1}{v^*}(1+a_{0}a_{1}\dots a_{n})$, respectively. In \cite{K}, Kollár proved that when $v^*=1$, the hypersurface $\{f_{0}=0\}\subset\mathbb{P}(\mathbf{v})$ is birational to the projective space $\mathbb{P}^{n-1}$. In addition, he showed that the rank of the middle homology group of the link is given by
$$\text{dim }H_{n-1}(L_{f_{0}},\mathbb{Q})=(-1)^{n+1}+v^{*}.$$
As $n$ is even, we have that when $v^{*}=1$, then the link $L_{f_{0}}$ is a rational homology sphere. Later, we will compute the torsion of $H_{n-1}(L_{f_{0}},\mathbb{Z})$ using the Orlik algorithm.

From now on, we say that the cycle polynomial $f_{0}$ satisfies the \textit{condition} $K1$ when $v^*=1$. Clearly, in this situation, we have that the weight vector is $\mathbf{v}=(v_{0}',v_{1}',\dots, v_{n}')$ and the degree is $m_{\alpha}=1+a_{0}a_{1}\dots a_{n}$.

\begin{lem}\label{lem:4.1.1}
    The following statements are equivalent:
 \begin{itemize}
     \item[(a)] $f_{0}$ satisfies the condition $K1$.
     \item[(b)] $\gcd(m_{\alpha},v_{i_{0}}')=1$ for some $i_{0}$.
     \item[(c)] $\gcd(m_{\alpha},v_{i}')=1$ for all $i=0,1,\dots,n$.
     \item[(d)] For $\{f_{0}=0\}\subset\mathbb{P}(\mathbf{v})$, the weighted projective space $\mathbb{P}(\mathbf{v})$ is well-formed.
 \end{itemize}
 In addition, if $n\geq 4$, then the fact that $f_{0}$ verifies the condition $K1$ implies that the weighted hypersurface $X_{f_{0}}\subset\mathbb{P}(\mathbf{v})$ is well-formed.
 \end{lem}
\begin{proof}
    First, we show that (a) implies (b). We can assume without loss of generality that $i_{0}=0$. If we suppose that $r=\gcd(m_{\alpha},v_{0})\neq1$, then by the equation $v_{n}'+a_{0}v_{0}'=m_{\alpha}$ in (\ref{eq:4.2}), we have $r\mid v_{n}'$. Following a similar process, we get $r\mid v_{n-1}',\dots, r\mid v_{1}'$, which is a contradiction. On the other hand, using again the equations in (\ref{eq:4.2}), we also obtain that (b) is equivalent to (c).

    Let us see that (c) implies (d). Indeed, we suppose that there exists some $j$ such that $\gcd(v_{0}',\dots,\hat{v_{j}'},\dots, v_{n}')=r\neq1$. We can assume $j=0$. Then $r\mid v_{1}'$ and $r\mid v_{2}'$. This implies that $r\mid m_{\alpha}$. Thus $\gcd(m_{\alpha},v_{1}')\neq1$, which is a contradiction.

    Finally, it is immediate that (d) implies (a).

    Now, from Lemma \ref{lem:2.1}, we will verify  $\gcd(v_{0}',\dots,\hat{v_{i}'},\dots,\hat{v_{j}'},\dots, v_{n})\mid m_{\alpha}$, for $n\geq4$. Here, we suppose without loss of generality that $i=0$. It implies that there exist two consecutive weights $v_{k}'$ and $v_{k+1}'$ different to $v_{0}'$ and $v_{j}'$. Thus, we have $\gcd(\hat{v_{0}'},v_{1},\dots, \hat{v_{j}'},\dots, v_{n}')\mid \gcd(v_{k}',v_{k+1}')$. Then, by the equation $v_{k}'+a_{k+1}v_{k+1}'=m_{\alpha}$ in (\ref{eq:4.2}), we obtain $$\gcd(v_{k}',v_{k+1}')\mid m_{\alpha}.$$ It follows that $$\gcd(\hat{v_{0}'},v_{1},\dots, \hat{v_{j}'},\dots, v_{n}')\mid m_{\alpha}.$$
\end{proof}

 Next, we use the Orlik algorithm to compute the middle Betti number of the link $L_{f_{0}}$ of a cycle polynomial $f_{0}$ that verifies the condition $K1$.
\begin{prop}\label{prop:4.1.2}
    Let $f_{0}$ be a cyclic polynomial defined as in (\ref{eq:4.1}) of degree $m_{\alpha}$ and weight vector $\mathbf{v}$ which verifies the condition $K1$. Then 
    $$H_{n-1}(L_{f_{0}},\mathbb{Z})=\mathbb{Z}_{m_{\alpha}}.$$
\end{prop}
\begin{proof}
    From \cite{K}, we have that the link $L_{f_{0}}$ is a rational homology sphere, which implies that the Betti number $b_{n-1}(L_{f_{0}})=0$. Now, we calculate the torsion. Since $f_{0}$ satisfies the condition $K1$, we have that $\gcd\left(m_{\alpha},v_{i}'\right)=1$ for all $i$. Thus, the values of $u_{i}$ and $v_{i}$ are
    $$u_{i}=\dfrac{m_{\alpha}}{\gcd\left(m_{\alpha},v_{i}'\right)}=m_{\alpha} \quad \text{ and } \quad v_{i}=\dfrac{v_{i}'}{\gcd\left(m_{\alpha},v_{i}'\right)}=v_{i}'.$$
    Now, we compute the pairs $(c_{i_{1},\dots,i_{s}},k_{i_{1},\dots,i_{s}})$. By definition, we obtain $$c_{\emptyset}=\gcd(u_{0},\dots,u_{n})=m_{\alpha} $$
    and $c_{i_{1},\dots, i_{s}}=1$ in other cases. Then, it is sufficient to compute $k_{\emptyset}$. As $n$ is even, we have $k_{\emptyset}=\epsilon_{n+1}=1$.  It follows that $H_{n-1}(L_{f_{0}},\mathbb{Z})=\mathbb{Z}_{m_{\alpha}}$.
\end{proof}

\subsection{Rational homology spheres from Thom-Sebastiani sums} Now, we will construct links that are rational homology $(2n-1)$-spheres which come from Thom-Sebastiani sums that are obtained adding a binomial block to a polynomial $f_{0}$ described as in (\ref{eq:4.1}). These polynomials have the form
\begin{align*}
    \text{Type I (Cycle-BP):} & \quad f=f_{0}(z_{0},\dots,z_{n})+z_{n+1}^{a_{n+1}}+z_{n+2}^{a_{n+2}},\\
    \text{Type II (Cycle-Chain):} & \quad f=f_{0}(z_{0},\dots,z_{n})+z_{n+1}^{a_{n+1}}+z_{n+1}z_{n+2}^{a_{n+2}},\\
    \text{Type III (Cycle-Cycle):} & \quad f=f_{0}(z_{0},\dots,z_{n})+z_{n+2}z_{n+1}^{a_{n+1}}+z_{n+1}z_{n+2}^{a_{n+2}}.
\end{align*}
Here, we assume that the polynomial $f_{0}$ satisfies the condition $K1$, where $m_{\alpha}$ and  $\mathbf{v}$ are its degree and weight vector, respectively. Then, for a polynomial $f$ of type I, II or III, we write its corresponding weight vector $\mathbf{w}$ and degree $d$ as
\begin{equation}\label{eq:4.4}
    \mathbf{w}=(w_{0},w_{1},\dots,w_{n},w_{n+1},w_{n+2})=(m_{\beta}\mathbf{v},m_{\alpha}v_{n+1},m_{\alpha}v_{n+2}) \quad  \text{and} \quad d=m_{\alpha}m_{\beta},
\end{equation}
so that these verify the following conditions:
\begin{equation}\label{eq:4.5}
\gcd(m_{\alpha},m_{\beta})=1 \quad \text{and} \quad \gcd(m_{\beta},v_{n+1})=\gcd(m_{\beta},v_{n+2})=1.
\end{equation}
Notice that the conditions in (\ref{eq:4.5}) imply that $\gcd(w_{0},w_{1},\dots, w_{n+2})=1$. On the other hand, the fact that the weight vector $\mathbf{w}$ in (\ref{eq:4.4}) admits a polynomial $f$ of type I, II or III of degree $d=m_{\alpha}m_{\beta}$ implies that $v_{n+1}$ and $v_{n+2}$ take some values in particular. In this sense, we have the following lemma:

\begin{lem} \label{lem:4.1.3}Let $\mathbf{w}$ be the weight vector associated to a polynomial $f$ of type I, II or III and degree $d=m_{\alpha}m_{\beta}$ so that $\mathbf{w}$ and $d$ verify (\ref{eq:4.5}). Then we have
\begin{itemize}
    \item[i)] If $f$ is type I, then $v_{n+1}=v_{n+2}=1$.
    \item[ii)] If $f$ is type II and $\mathbf{w}$ does not admit polynomials of type I, then $v_{n+1}=1$ and $v_{n+2}\neq 1$.
    \item[iii)] If $f$ is type III and $\mathbf{w}$ does not admit polymonials type II, then $v_{n+1}\neq 1$ and $v_{n+2}\neq 1$.
\end{itemize}
\end{lem}
\begin{proof}
\begin{itemize}
    \item[i)] Let $f$ be a polynomial of type I:
    $$f=f_{0}(z_{0},\dots,z_{n})+z_{n+1}^{a_{n+1}}+z_{n+2}^{a_{n+2}}$$
    of degree $d=m_{\alpha}m_{\beta}$ and weight vector $\mathbf{w}$ as in (\ref{eq:4.4}). For the weights $w_{n+1}$ and $w_{n+2}$ we have
    $$a_{n+1}w_{n+1}=d \Leftrightarrow a_{n+1}v_{n+1}=m_{\beta} \qquad \text{ and }\qquad a_{n+2}w_{n+2}=d \Leftrightarrow a_{n+2}v_{n+2}=m_{\beta}.$$
    Since $\gcd(m_{\beta},v_{n+1})=\gcd(m_{\beta},v_{n+2})=1$, we obtain $v_{n+1}=v_{n+2}=1$.
    \item[ii)] We consider the polynomial $f$ of type II:
    $$f=f_{0}(z_{0},\dots,z_{n})+z_{n+1}^{a_{n+1}}+z_{n+1}z_{n+2}^{a_{n+2}}$$
    of degree $d=m_{\alpha}m_{\beta}$ and weight vector $\mathbf{w}$ as in (\ref{eq:4.4}). As in item i), we obtain $v_{n+1}=1$. On the other hand, from the monomial $z_{n+1}z_{n+2}^{a_{n+2}}$, we have
    $$w_{n+1}+a_{n+2}w_{n+2}=d \Leftrightarrow 1+a_{n+2}v_{n+2}=m_{\beta}.$$
    If we suppose $v_{n+2}=1$, then we obtain the equality $1+a_{n+2}=m_\beta$. Now, we define the polynomial of type I:
    $$\tilde{f}=f_{0}+z_{n+1}^{a_{n+1}}+z_{n+2}^{1+a_{n+2}}.$$
    After some calculations, we have:
    $$(1+a_{n+2})w_{n+2}=(1+a_{n+2})m_{\alpha}v_{n+2}=m_{\beta}m_{\alpha}=d.$$
    This implies that the polynomial $\tilde{f}$ of type I has degree $d=m_{\alpha}m_{\beta}$ and its weight vector $\mathbf{w}$ is given as in (\ref{eq:4.4}), which is a contradiction. Thus $v_{n+2}\neq 1$.
    \item[iii)] Let $f$ be a polynomial of type III:
    $$f=f_{0}+z_{n+2}z_{n+1}^{a_{n+1}}+z_{n+1}z_{n+2}^{a_{n+2}}$$
    of degree $d=m_{\alpha}m_{\beta}$ and weight vector $\mathbf{w}$ as in (\ref{eq:4.4}). Let us see that $v_{n+1}\neq1$ and $v_{n+2}\neq1$. We prove this by contradiction. Without loss of generality, we assume $v_{n+1}=1$. Now, we define the polynomial
    $$\tilde{f}=f_{0}+z_{n+1}^{m_{\beta}}+z_{n+1}z_{n+2}^{a_{n+2}}.$$
    Clearly, $\tilde{f}$ is a polynomial of type II that has degree $d$ and weight vector $\mathbf{w}$, which is a contradiction. So  $v_{n+1}\neq 1$.
\end{itemize}
\end{proof}

Next, we show that the links that correspond to polynomials of type I, II or III are rational homology $(2n+3)$-spheres. 
\begin{prop}\label{prop:4.1.4}
     Given the polynomial $f=f_{0}+g(z_{n+1},z_{n+2})$ of type I, II or III of degree $d=m_{\alpha}m_{\beta}$ and weight vector $\mathbf{w}=(m_{\beta}\mathbf{v},m_{\alpha}v_{n+1},m_{\alpha}v_{n+2})$ satisfying the conditions in (\ref{eq:4.5}), where $f_{0}$ is a cycle polynomial of degree $m_{\alpha}$ and weight vector $\mathbf{v}$ verifying the condition $K1$. Then the link $L_{f}$ is a rational homology $(2n+3)$-sphere.
\end{prop}

\begin{proof}
    First, we compute the Betti number $b_{n+1}$. For this, we calculate the Alexander polynomial of link $L_{f}$. For $i=0,1,\dots, n$, we have
    $$u_{i}=\dfrac{d}{\gcd(d,w_{i})}=\dfrac{m_{\alpha}m_{\beta}}{\gcd\left(m_{\alpha}m_{\beta},m_{\beta}v_{i}'\right)}=\dfrac{m_{\alpha}}{\gcd\left(m_{\alpha},v_{i}'\right)}.$$
    As $\gcd(m_{\alpha},v_{i}')=1$, we obtain $u_{i}=m_{\alpha}$, for each $i=0,1,\dots,n$. On the other hand, for $i\in\{n+1,n+2\}$, we have
    $$u_{i}=\dfrac{d}{\gcd(d,w_{i})}=\dfrac{m_{\alpha}m_{\beta}}{\gcd(m_{\alpha}m_{\beta},m_{\alpha}v_{i})}=\dfrac{m_{\beta}}{\gcd(m_{\beta},v_{i})}.$$
    As $\gcd(m_{\beta},v_{n+1})=\gcd(m_{\beta},v_{n+2})=1$, we conclude that $u_{i}=m_{\beta}$ for $i=n+1,n+2$. Now, we compute the values $v_{i}$'s. For $i=0,1,\dots,n$, we have
    $$v_{i}=\dfrac{w_{i}}{\gcd(d,w_{i})}=\dfrac{m_{\beta}v_{i}'}{\gcd\left(m_{\alpha}m_{\beta},m_{\beta}v_{i}'\right)}=v_{i}'.$$
    Moreover, for $i=n+1$ or $n+2$, we have the values
    $$\dfrac{w_{i}}{\gcd(d,w_{i})}=\dfrac{m_{\alpha}v_{i}}{\gcd(m_{\alpha}m_{\beta},m_{\alpha}v_{i})}=v_{i}.$$
    Then, we obtain the divisor of $\Delta(t)$:
    \begin{align*}
        \operatorname{div}\Delta & = \left( \frac{1}{v_{0}'}\Lambda_{m_{\alpha}}-\Lambda_{1}\right)\dots \left( \frac{1}{v_{n}'}\Lambda_{m_{\alpha}}-\Lambda_{1}\right)\left( \frac{1}{v_{n+1}}\Lambda_{m_{\beta}}-\Lambda_{1}\right)\left( \frac{1}{v_{n+2}}\Lambda_{m_{\beta}}-\Lambda_{1}\right)\\
        &= \left(\tau_{\alpha}\Lambda_{m_{\alpha}}-\Lambda_{1}\right)\left(\tau_{\beta}\Lambda_{m_{\beta}}+\Lambda_{1}\right),
    \end{align*}
    where 
    $$\tau_{\alpha}=\frac{m_{\alpha}^{n}-\left(v_{0}'+\cdots+v_{n}'\right)m_{\alpha}^{n-1}+\cdots+\sum_{j=0}^{n}\left(v_{0}'\dots \hat{v}_{j}'\dots v_{n}'\right)}{\prod_{i=0}^{n}v_{i}'}$$
    and 
    $$\tau_{\beta}=\dfrac{m_{\beta}}{v_{n+1}v_{n+2}}-\dfrac{1}{v_{n+1}}-\dfrac{1}{v_{n+2}}.$$
    From (\ref{eq:4.2}) and $m_{\alpha}=1+a_{0}a_{1}\dots a_{n}$, we have 
    \begin{equation}\label{eq:4.6}
        \prod_{i=0}^{n}\left(m_{\alpha}-v_{i}'\right)=\prod_{i=0}^{n}(a_{i}v_{i}')=(m_{\alpha}-1)\prod_{i=0}^{n}v_{i}'.
    \end{equation}
    On the other hand, we can express $\tau_{\alpha}$ as
    \begin{equation}\label{eq:4.7}
    \tau_{\alpha}=\dfrac{\prod_{i=0}^{n}\left(m_{\alpha}-v_{i}'\right)+\prod_{i=0}^{n}v_{i}'}{m_{\alpha}\prod_{i=0}^{n}v_{i}'}.
    \end{equation}
    Putting (\ref{eq:4.6}) in (\ref{eq:4.7}), we obtain $\tau_{\alpha}=1$. Computing and simplifying, we obtain the following expression for the divisor of $\Delta$:
    $$\operatorname{div}\Delta=\tau_{\beta}\Lambda_{d}-\tau_{\beta}\Lambda_{m_{\beta}}+\Lambda_{m_{\alpha}}-\Lambda_{1}.$$
    Thus, the Alexander polynomial of the link is given by
    $$\Delta(t)=\dfrac{(t^{d}-1)^{\tau_{\beta}}(t^{m_{\alpha}}-1)}{(t^{m_{\beta}}-1)^{\tau_{\beta}}(t-1)}.$$
    Notice that $\Delta(1)\neq 0$. Therefore, we have that $L_{f}$ is a rational homology $(2n+3)$-sphere.

    Next, we calculate the torsion of $H_{n+1}(L_{f},\mathbb{Z})$. For this, we will compute the pairs $(c_{i_{1},\dots,i_{s}},k_{i_{1},\dots,i_{s}})$ of the Orlik algorithm. Here we obtain $c_{\emptyset}=1$, $c_{n+1,n+2}=m_{\alpha}$, $c_{0,1,\dots,n}=m_{\beta}$ and $c_{i_{1},\dots,i_{s}}=1$ in other cases. Thus, it is enough to compute $k_{n+1,n+2}$ and $k_{0,1,\dots,n}$. By definition, as $n$ is even, we have $k_{0,1,\dots,n}=0$. On the other hand, we compute
    $$k_{n+1,n+2}=1-\dfrac{1}{v_{n+1}}-\dfrac{1}{v_{n+2}}+\dfrac{u_{n+1}u_{n+2}}{v_{n+1}v_{n+2}\gcd(u_{n+1},u_{n+2})}=\tau_{\beta}+1.$$
    Thus, we obtain
    \begin{equation}\label{eq:4.8}
        H_{n+1}(L_{f},\mathbb{Z})=\left(\mathbb{Z}_{m_{\alpha}}\right)^{\tau_{\beta}+1}.
    \end{equation}
\end{proof}

\begin{lem}\label{rem:4.1.5}
If $f$ is a polynomial of type $I, II$ or $III$ satisfying the conditions given in (\ref{eq:4.5}), then the orbifold $X_f\subset \mathbb{P}(\mathbf{w})$  cut out by $f$ is well-formed. 
\end{lem}
\begin{proof}
According to Lemma \ref{lem:2.1}, it is enough to show the following two conditions:  
\begin{itemize}
    \item[(1)] $\mathbb{P}(\mathbf{w})$ is well-formed for  the weight vector $\mathbf{w}=(m_{\beta}\mathbf{v},m_{\alpha}v_{n+1},m_{\alpha}v_{n+2})$ established in Proposition \ref{prop:4.1.4}.  For this, we must verify  $\gcd(w_{0},\dots,\hat{w}_{j},\dots,w_{n+2})=1$. First, we assume $j\in\{0,1,\dots,n\}$. Since $\gcd(w_{0},\dots,\hat{w}_{j},\dots,w_{n})=m_{\beta}$, we obtain 
    \begin{align*}
\gcd(w_{0},\dots,\hat{w}_{j},\dots,w_{n+2}) & =\gcd(w_{n+1},w_{n+2},\gcd(w_{0},\dots,\hat{w}_{j},\dots,w_{n}))    \\
    & =\gcd(w_{n+1},w_{n+2},m_{\beta}).
    \end{align*}
    From (\ref{eq:4.5}), we have $\gcd(w_{n+1},w_{n+2},m_{\beta})=1$. On the other hand, we suppose without loss of generality that $j=n+1$. Then
$$\gcd(w_{0},\dots,w_{n},w_{n+2})=\gcd(w_{n+2},\gcd(w_{0},\dots,w_{n}))=\gcd(w_{n+2},m_{\beta}).$$
    Again, by conditions in (\ref{eq:4.5}), we have $\gcd(w_{n+2},m_{\beta})=1$.
    \item[(2)] We need to verify $$\gcd(w_{0},\dots,\hat{w}_{i},\dots,\hat{w}_{j},\dots,w_{n+2})\mid d.$$ 
First, if $i,j\in\{0,1,\dots,n\}$. Here we have
\begin{equation}\label{eq:4.9}
\gcd(w_{0},\dots,\hat{w}_{i},\dots,\hat{w}_{j},\dots,w_{n+2})=\gcd(w_{n+1},w_{n+2},\gcd(w_{0},\dots,\hat{w}_{i},\dots,\hat{w}_{j},\dots,w_{n})).
\end{equation}
Also, as 
$\gcd(w_{0},\dots,\hat{w}_{i},\dots,\hat{w}_{j},\dots,w_{n})=m_{\beta}\gcd(v_{0}',\dots,\hat{v_{i}'},\dots,\hat{v_{j}'},\dots,v_{n}')$ and 
\newline  
$\gcd(v_{0}',\dots,\hat{v_{i}'},\dots,\hat{v_{j}'},\dots,v_{n}')\mid m_{\alpha}$ by Lemma \ref{lem:4.1.1}, we obtain
$$\gcd(w_{n+1},w_{n+2},\gcd(w_{0},\dots,\hat{w}_{i},\dots,\hat{w}_{j},\dots,w_{n}))=\gcd(v_{0}',\dots,\hat{v_{i}'},\dots,\hat{v_{j}'},\dots,v_{n}'),$$
which divides $m_{\alpha}$ so it divides $d$. On the other hand, if we take $i\in\{0,\dots,n\}$ and $j=n+1$, then we have
\begin{align*}
    \gcd(w_{0},\dots,\hat{w}_{i},\dots,w_{n},w_{n+2}) & = \gcd(w_{n+2},\gcd(w_{0},\dots,\hat{w}_{i},\dots,w_{n})) \\
    & = \gcd(w_{n+2},m_{\beta}\gcd(v_{0}',\dots,\hat{v_{i}'},\dots,v_{n}')).
\end{align*}
By Lemma 4.1.1, we have $\gcd(v_{0}',\dots,\hat{v_{i}'},\dots,v_{n}')=1$. This implies that 
$$ \gcd(w_{0},\dots,\hat{w}_{i},\dots,w_{n},w_{n+2})=\gcd(w_{n+2},m_{\beta})=\gcd(m_{\alpha}v_{n+2},m_{\beta})=1.$$
Finally, if $i=n+1$ and $j=n+2$, then $\gcd(w_{0},\dots,w_{n})=m_{\beta}$, which divides $d$. 
\end{itemize}
Thus, we conclude $X_{f}$ is well-formed.
\end{proof}

\begin{rem}  Using formula \ref{eq:3.6}, we can compute the Milnor number of $L_{f}$ for $f$ as in Proposition \ref{prop:4.1.4}:
    \begin{align*}
        \mu(L_{f})=\prod_{j=0}^{n+2}\left(\dfrac{d}{w_{j}}-1\right) & = a_{0}a_{1}\dots a_{n}\left(\dfrac{m_{\beta}}{v_{n+1}}-1\right)\left(\dfrac{m_{\beta}}{v_{n+2}}-1\right) \\
        & = a_{0}a_{1}\dots a_{n}(m_{\beta}(\tau_{\beta}-1)+1).
    \end{align*}
\end{rem}
    
\section{Berglund-Hübsch rule and the topology of links}
In this section, we explain how the Berglund-Hübsch rule works. Moreover, assuming further conditions, we prove that this method keeps invariant the topological property of being rational homology sphere of the links. These links come from the polynomials described in the preceding section. The results of this section are a generalization of the main results in \cite{CGL}. 

Given an invertible polynomial 
$$f(x_{1},x_{2},\dots,x_{n})=\sum_{j=1}^{n}\prod_{i=1}^{n}x_{i}^{a_{ij}}$$
we recall that the exponents of each monomial arranged in rows define a \textit{matrix of exponents} $A=[a_{ij}]_{n\times n}$ which is invertible. The \textit{Berglund-Hübsch transpose rule}, which we abbreviate as \textit{B-H rule}, consists in taking the transpose matrix $A^{T}=[a_{ji}]_{n\times n}$ of $A$ and construct a new polynomial $f^{T}$ which we call \textit{dual polynomial}. Originally, this rule was developed by Berglund, Hübsch and Krawitz to exhibit examples of mirror pairs of Calabi-Yau orbifolds  (see \cite{BH} and \cite{Kr}). Here, the dual polyomial is given by
$$f^{T}(x_{1},x_{2},\dots,x_{n})=\sum_{j=1}^{n}\prod_{i=1}^{n}x_{i}^{a_{ji}}$$
Easily we can note that $f^{T}$ is also invertible.

In the lines below, we describe the links that come from a dual polynomial $f^{T}$, when $f=f_{0}$ or $f$ is of type I, II or III. Before, we show the following technical result:

\begin{lem}\label{lem:4.2.1}
    Consider the positive integers $a_{0},a_{1},\dots a_{n}$ and $m_{\alpha}=1+a_{0}a_{1}\dots a_{n}$. A solution for the system of linear equations 
    \begin{equation}\label{eq:4.10}
        a_{0}\tilde{v}_{0}+\tilde{v}_{1}=m_{\alpha}, \quad a_{1}\tilde{v}_{1}+\tilde{v}_{2}=m_{\alpha},\quad \dots \quad a_{n}\tilde{v}_{n}+\tilde{v}_{0}=m_{\alpha}
    \end{equation}
    is given by the values $\tilde{v}_{i}$ defined by
    \begin{equation}\label{eq:4.11}
        \tilde{v}_{i}=1-a_{i-1}+a_{i-1}a_{i-2}-\dots+a_{i-1}a_{i-2}\cdots a_{i-n}, 
    \end{equation}
    where the subscripts are taken$\mod (n+1)$. In addition, if we consider the values $v_{i}'$ given in (\ref{eq:4.3}), that is
    \begin{equation*}
    v_{i}'=1-a_{i+1}+a_{i+1}a_{i+2}-\dots +a_{i+1}a_{i+2}\dots a_{i+n-1}a_{i+n},
\end{equation*}
    then it follows
    $$\sum_{i=0}^{n}\tilde{v}_{i}=\sum_{i=0}^{n}v_{i}'.$$
 \end{lem}
\begin{proof}
    Indeed, for each $i$, we compute
    $$a_{i}\tilde{v}_{i}=a_{i}-a_{i}a_{i-1}+a_{i}a_{i-1}a_{i-2}-\cdots-a_{i}a_{i-1}a_{i-2}\cdots a_{i-n+1}+a_{i}a_{i-1}a_{i-2}\cdots a_{i-n}.$$
    Adding $\tilde{v}_{i+1}=1-a_{i}+a_{i}a_{i-1}-\cdots+a_{i}a_{i-1}\cdots a_{i-n+1}$ in the equality above, we obtain
    $$a_{i}\tilde{v}_{i}+\tilde{v}_{i+1}=1+a_{i}a_{i-1}a_{i-2}\cdots a_{i-n}=1+a_{0}a_{1}\cdots a_{n}=m_{\alpha}.$$
    Finally, $\sum_{i=0}^{n}v_{i}=\sum_{i=0}^{n}\tilde{v}_{i}$ is valid since we can write for each $i$ and $1\leq j\leq n$:
    $$a_{i-1}a_{i-2}\cdots a_{i-j+1}a_{i-j}=a_{i-j}a_{i-j+1}\cdots a_{i-j+(j-1)}.$$
\end{proof}
\medskip

\subsection{Berglund-Hübsch for cyclic polynomials} We begin studying the effect of the B-H rule on the cycle polynomial $f_{0}$. 
\begin{theo}\label{prop:4.2.2}
    Let $f_{0}$ be a cycle polynomial of degree $m_{\alpha}$ and weight vector $\mathbf{v}$ satisfying the condition $K1$. Then the dual polynomial $f_{0}^{T}$ of $f_{0}$ obtained by the B-H rule is also a cycle polynomial. Moreover, if $f_{0}^{T}$ satisfies the condition $K1$ then the link  $L_{f_{0}^{T}}$ is a rational homology $(2n-1)$-sphere and
    \begin{equation}\label{eq:4.12}
        H_{n-1}(L_{f_{0}^{T}},\mathbb{Z})=\mathbb{Z}_{m_{\alpha}}.
    \end{equation}
\end{theo}
\begin{proof}
We consider the matrix of exponents $A_{0}$ of $f_{0}$:
$$A_{0}=\begin{bmatrix}
    a_{0} & 0 & 0 & \cdots & 0 & 0 & 1 \\
    1 & a_{1} & 0 & \cdots & 0 & 0 & 0 \\
    0 & 1 & a_{2} & \cdots& 0 & 0 & 0\\
    \vdots & \vdots & \vdots & \ddots & \vdots & \vdots & \vdots \\
    0  & 0 & 0 & \cdots & 1 & a_{n-1} & 0\\
    0 & 0 & 0 & \cdots & 0 & 1 & a_{n}
\end{bmatrix}.$$
Taking the transpose matrix $A_{0}^{T}$, we verify that the dual polynomial is cycle type:
\begin{equation}\label{eq:4.13}
    f_{0}^{T}= z_{0}^{a_{0}}z_{1}+z_{1}^{a_{1}}z_{2}+\cdots +z_{n-1}^{a_{n-1}}z_{n}+z_{n}^{a_{n}}z_{0}.
\end{equation}
Now, we suppose that $f_{0}^{T}$ verifies the condition $K1$. To determine the degree $\tilde{m}_{\alpha}$ and the weight vector $\tilde{\mathbf{v}}$ of $f_{0}^{T}$, we solve the following matrix equation
$$A_{0}^{T}\tilde{\mathbf{v}}^{T}=D^{T},$$
where $D=\begin{bmatrix}
    \tilde{m}_{\alpha} & \tilde{m}_{\alpha} & \cdots & \tilde{m}_{\alpha}
\end{bmatrix}$. From Lemma \ref{lem:4.2.1}, we obtain the solutions $\lambda(\tilde{v}_{0},\tilde{v}_{1},\dots, \tilde{v}_{n},\tilde{m}_{\alpha})$, where $\lambda\in\mathbb{R}$, $\tilde{m}_{\alpha}=1+a_{0}a_{1}\dots a_{n}$ and each $\tilde{v}_{i}$ is defined as in (\ref{eq:4.11}):
$$\tilde{v}_{i}=1-a_{i-1}+a_{i-1}a_{i-2}-\dots+a_{i-1}a_{i-2}\cdots a_{i-n}.$$
As $f_{0}^{T}$ satisfies condition $K1$, then the weight vector and the degree of $f_{0}^{T}$ are given by $\tilde{\mathbf{v}}=(\tilde{v}_{0},\dots, \tilde{v}_{n})$ and $\tilde{m}_{\alpha}=1+a_{0}a_{1}\dots a_{n}$. Finally, applying Proposition \ref{prop:4.1.2} to the polynomial $f_{0}^{T}$, we obtain that the link $L_{f_{0}^{T}}$ is a rational homology $(2n-1)$-sphere and  
\begin{equation}\label{4.14}
    H_{n-1}(L_{f_{0}^{T}},\mathbb{Z})=\mathbb{Z}_{m_{\alpha}}.
\end{equation}
\end{proof}

\begin{rem}\label{rem:4.2.3}
    Since $f_{0}^{T}$ verifies the condition $K1$, we have that the weighted projective space $\mathbb{P}(\mathbf{\tilde{v}})$ is well-formed. Moreover, for $n\geq 4$, the weighted hypersurface $X_{f_{0}^{T}}\subset\mathbb{P}(\mathbf{\tilde{v}})$ is well-formed.
\end{rem}

\subsection{Berglund-Hübsch for Thom-Sebastiani sums}

Now, we study what happens with the link associated to the dual polynomial $f^{T}$ when $f$ is of type I, II or III. First we deal with polynomials of type I and III.

\begin{theo}\label{prop:4.2.4}
    Let $$f=f_{0}+g(z_{n+1},z_{n+2})$$ be a polynomial of type I or III of degree $d=m_{\alpha}m_{\beta}$ and weight vector $\mathbf{w}=(m_{\beta}\mathbf{v},m_{\alpha}v_{n+1},m_{\alpha}v_{n+2})$ verifying \ref{eq:4.5}), where  $f_{0}$ satisfies the condition $K1$ and has degree $m_{\alpha}$ and weight vector $\mathbf{v}$. If we denote by $f^T$ as the dual polynomial of $f$, then 
    $$f^{T}=f_{0}^{T}+g(z_{n+1},z_{n+2}).$$ Moreover, if $f_{0}^{T}$ satisfies the condition $K1$, then the link  $L_{f^{T}}$ is a rational homology $(2n+3)$-sphere and
    $$H_{n+1}(L_{f},\mathbb{Z})=H_{n+1}(L_{f^{T}},\mathbb{Z}).$$
\end{theo}
\begin{proof}
    First, we see the case when $f=f_{0}+z_{n+1}^{a_{n+1}}+z_{n+2}^{a_{n+2}}$. Here, by Lemma \ref{lem:4.1.3} we have $v_{n+1}=v_{n+2}=1$. It means $\mathbf{w}=(m_{\beta}\mathbf{v},m_{\alpha},m_{\alpha})$. Then we have the matrix of exponents of $f$
    $$A_{I}=\begin{bmatrix}
    a_{0} & 0 & \cdots  & 0 & 1 & 0 & 0 \\
    1 & a_{1} &  \cdots  & 0 & 0 & 0 & 0 \\
    \vdots & \vdots & \ddots & \vdots & \vdots & \vdots & \vdots \\
    0  & 0 & \cdots & a_{n-1} & 0 & 0 & 0\\
    0 & 0 &  \cdots & 1 & a_{n} & 0 & 0 \\
    0 & 0 & \cdots & 0 & 0  & a_{n+1} & 0 \\
    0 & 0 & \cdots & 0 & 0  & 0 & a_{n+2} 
\end{bmatrix}.$$
Applying the B-H rule, we obtain the dual polynomial 
$$f^{T}=f_{0}^{T}+z_{n+1}^{a_{n+1}}+z_{n+2}^{a_{n+2}},$$
where $f_{0}^{T}$ is given as in (\ref{eq:4.13}). Now, we suppose that $f_{0}^{T}$ verifies the condition $K1$. To determine the weight vector $\mathbf{\tilde{w}}$ and the degree $\tilde{d}$ of $f^{T}$, we solve the matrix equation
$$A_{I}^{T}\mathbf{\tilde{w}}=D^{T},$$
where $D=\begin{bmatrix}\tilde{d} & \tilde{d} & \cdots & \tilde{d}\end{bmatrix} $. Using Lemma \ref{lem:4.2.1}, we have the solutions 
$$\lambda(m_{\beta}\tilde{v}_{0},m_{\beta}\tilde{v}_{1},\dots, m_{\beta}\tilde{v}_{n},m_{\alpha},m_{\alpha},\tilde{d}),$$ where $\lambda\in\mathbb{R}$, $\tilde{d}=m_{\alpha}m_{\beta}=(1+a_{0}a_{1}\dots a_{n})m_{\beta}$ and each $\tilde{v}_{i}$ is defined as in (\ref{eq:4.11}):
$$\tilde{v}_{i}=1-a_{i-1}+a_{i-1}a_{i-2}-\cdots + a_{i-1}a_{i-2}\dots a_{i-n}.$$
Since $f_{0}^{T}$ verifies the condition $K1$, we obtain the weight vector $\tilde{\mathbf{w}}$ and the degree $\tilde{d}$ of $f^{T}$:
$$\tilde{\mathbf{w}}=(m_{\beta}\tilde{v}_{0},m_{\beta}\tilde{v}_{1},\dots, m_{\beta}\tilde{v}_{n},m_{\alpha},m_{\alpha}) \quad \text{ and } \quad \tilde{d}=m_{\alpha}m_{\beta}.$$
Then, applying Proposition \ref{prop:4.1.4}, we conclude that the link $L_{f^{T}}$ is a rational homology $(2n+3)$-sphere. Moreover, 
$$H_{n+1}(L_{f^T},\mathbb{Z})=(\mathbb{Z}_{m_{\alpha}})^{\tau_{\beta}+1},$$
where $\tau_{\beta}=m_{\beta}-1$ for the cycle-BP polynomial $f^{T}$.

On the other hand, for a polynomial type III, $$f=f_{0}+z_{n+2}z_{n+1}^{a_{n+1}}+z_{n+1}z_{n+2}^{a_{n+2}},$$ the process is similar. Here, the weight vector $\mathbf{\tilde{w}}$ and degree $\tilde{d}$ of the dual polynomial $f^{T}$ are given by
\begin{equation*}
    \mathbf{\tilde{w}}=(m_{\beta}\tilde{v}_{0},m_{\beta}\tilde{v}_{1},\dots,m_{\beta}\tilde{v}_{n},m_{\alpha}v_{n+1},m_{\alpha}v_{n+2}) \quad \text{ and } \quad \tilde{d} = m_{\alpha}m_{\beta}.
\end{equation*}
Then, using again Proposition \ref{prop:4.1.4}, we have that the link $L_{f^{T}}$ is a rational homology $(2n+3)$-sphere and
$$H_{n+1}(L_{f^T},\mathbb{Z})=(\mathbb{Z}_{m_{\alpha}})^{\tau_{\beta}+1},$$
where $$\tau_{\beta}=\dfrac{m_{\beta}}{v_{n+1}v_{n+2}}-\dfrac{1}{v_{n+1}}-\dfrac{1}{v_{n+2}}.$$
\end{proof}
\begin{rem}\label{rem:4.2.5}
    Following the same idea as (1) and (2) in Lemma \ref{rem:4.1.5}, we conclude that $\mathbb{P}(\mathbf{\tilde{w}})$ is well-formed. Moreover, for $n\geq 4$, the weighted hypersurface $X_{f^{T}}\subset\mathbb{P}(\mathbf{\tilde{w}})$ is well-formed.
\end{rem}
\medskip

Now we study the duals of  Thom-Sebastiani sums of a cycle and a chain block that have the form of polynomials of type II.  

\begin{theo}\label{prop:4.2.6}
    Let $$f=f_{0}+z_{n+1}^{a_{n+1}}+z_{n+1}z_{n+2}^{a_{n+2}}$$ be a polynomial of type II 
    with dual polynomial given by $$f^{T}=f_{0}^{T}+z_{n+2}z_{n+1}^{a_{n+1}}+z_{n+2}^{a_{n+2}}.$$
    Let us assume the following:
 \begin{itemize}   
 \item The degree of $f$ is $d=m_{\alpha}m_{\beta}$ and its weight vector is  $\mathbf{w}=(m_{\beta}\mathbf{v},m_{\alpha}v_{n+1},m_{\alpha}v_{n+2}).$
    \item  $f_{0}$ verifies the condition $K1$  and 
    \item $\gcd(m_{\beta},a_{n+2}-1)=1.$
 \end{itemize}   
     Then  the link $L_{f^{T}}$ is a rational homology $(2n+3)$-sphere. 
\end{theo}
\begin{proof}
    For the polynomial $f$, we have its  matrix of exponents 
   $$A_{II}=\begin{bmatrix}
    a_{0} & 0 & \cdots  & 0 & 1 & 0 & 0 \\
    1 & a_{1} &  \cdots  & 0 & 0 & 0 & 0 \\
    \vdots & \vdots & \ddots & \vdots & \vdots & \vdots & \vdots \\
    0  & 0 & \cdots & a_{n-1} & 0 & 0 & 0\\
    0 & 0 &  \cdots & 1 & a_{n} & 0 & 0 \\
    0 & 0 & \cdots & 0 & 0  & a_{n+1} & 0 \\
    0 & 0 & \cdots & 0 & 0  & 1 & a_{n+2} 
\end{bmatrix}.$$
Taking the transpose $A_{II}^{T}$, we obtain the dual polynomial
$f^{T}=f_{0}^{T}+z_{n+1}^{a_{n+1}}z_{n+2}+z_{n+2}^{a_{n+2}}$. Now, since $f_{0}$ satisfies the condition $K1$, it follows from Theorem  \ref{prop:4.2.2} and   Lemma \ref{lem:4.1.1},  that  $f_0^T$ satisfies  
\begin{itemize}
\item $f_{0}^{T}$ is a cyclic polynomial satisfying the condition $K1$
\item the weight vector and the degree of $f_0^T$ are given by $\tilde{\mathbf{v}}=\left(\tilde{v}_0, \ldots, \tilde{v}_n\right)$ 
\item  $\tilde{m}_\alpha=1+a_0 a_1 \ldots a_n,$ and  
\item $\gcd(m_{\alpha},\tilde{v}_{i})=1$  for  $i=0,1,\dots,n.$  
\end{itemize}

In order to find the weight vector $\mathbf{\tilde{w}}$ and the degree $\tilde{d}$ of $f^{T}$, we solve the matrix equation
$$A_{II}^{T}\mathbf{\tilde{w}}^{T}=D^{T},$$
where $D=\begin{bmatrix}
    \tilde{d} & \tilde{d} & \cdots & \tilde{d}
\end{bmatrix}$. A solution for this equation is given by 
$$\mathbf{\tilde{w}}=(\tilde{w}_{0},\dots,\tilde{w}_{n+2})=(a_{n+2}m_{\beta}\tilde{v}_{0},a_{n+2}m_{\beta}\tilde{v}_{1},\dots, a_{n+2}m_{\beta}\tilde{v}_{n},m_{\alpha}(a_{n+2}-1),m_{\alpha}m_{\beta})$$
and $\tilde{d}=a_{n+2}m_{\alpha}m_{\beta}$. We notice that in order to compute the values $u_{i}$ and $v_{i}$ in the formula of the Alexander polynomial, it is enough to obtain a vector that is parallel to the weight vector. Now, we compute these values. 

For  $i=0,1,\dots,n$ we have:
$$u_{i}=\dfrac{\tilde{d}}{\gcd(\tilde{d},\tilde{w}_{i})}=\dfrac{a_{n+2}m_{\alpha}m_{\beta}}{\gcd(a_{n+2}m_{\alpha}m_{\beta},a_{n+2}m_{\beta}\tilde{v}_{i})}=m_{\alpha}$$
and
$$v_{i}=\dfrac{\tilde{w}_{i}}{\gcd(\tilde{d},\tilde{w}_{i})}=\dfrac{a_{n+2}m_{\beta}\tilde{v}_{i}}{\gcd(a_{n+2}m_{\alpha}m_{\beta},a_{n+2}m_{\beta}\tilde{v}_{i})}=\tilde{v}_{i}$$
In case  $i=n+1,$ since   $\gcd(m_{\beta},a_{n+2}-1)=1$ we have   
$$u_{n+1}=\dfrac{\tilde{d}}{\gcd(\tilde{d},\tilde{w}_{n+1})}=\dfrac{a_{n+2}m_{\alpha}m_{\beta}}{\gcd(a_{n+2}m_{\alpha}m_{\beta},m_{\alpha}(a_{n+2}-1))}=a_{n+2}m_{\beta}$$
and 
$$v_{n+1}=\dfrac{\tilde{w}_{n+1}}{\gcd(\tilde{d},\tilde{w}_{n+1})}=\dfrac{m_{\alpha}(a_{n+2}-1)}{\gcd(a_{n+2}m_{\alpha}m_{\beta},m_{\alpha}(a_{n+2}-1))}=a_{n+2}-1.$$
Finally, for $i=n+2$ we have
$$u_{n+2}=\dfrac{\tilde{d}}{\gcd(\tilde{d},\tilde{w}_{n+2})}=\dfrac{a_{n+2}m_{\alpha}m_{\beta}}{\gcd(a_{n+2}m_{\alpha}m_{\beta},m_{\alpha}m_{\beta})}=a_{n+2}$$
and 
$$v_{n+2}=\dfrac{\tilde{w}_{n+2}}{\gcd(\tilde{d},\tilde{w}_{n+2})}=\dfrac{m_{\alpha}m_{\beta}}{\gcd(a_{n+2}m_{\alpha}m_{\beta},m_{\alpha}m_{\beta})}=1.$$
Then the divisor of the Alexander polynomial is
$$\operatorname{div} \Delta = \left(\dfrac{1}{\tilde{v}_{0}}\Lambda_{m_{\alpha}}-\Lambda_{1}\right)\dots\left(\dfrac{1}{\tilde{v}_{n}}\Lambda_{m_{\alpha}}-\Lambda_{1}\right)\left(\dfrac{1}{a_{n+2}-1}\Lambda_{a_{n+2}m_{\beta}}-\Lambda_{1}\right)\left(\Lambda_{a_{n+2}}-\Lambda_{1}\right).$$
Using a similar argument as the one used in Proposition \ref{prop:4.1.4}, we have
$$\left(\dfrac{1}{\tilde{v}_{0}}\Lambda_{m_{\alpha}}-\Lambda_{1}\right)\dots\left(\dfrac{1}{\tilde{v}_{n}}\Lambda_{m_{\alpha}}-\Lambda_{1}\right)=\Lambda_{m_{\alpha}}-\Lambda_{1}.$$
Replacing above, we obtain
\begin{align*}
    \operatorname{div} \Delta & = (\Lambda_{m_{\alpha}}-\Lambda_{1})\left(\dfrac{1}{a_{n+2}-1}\Lambda_{a_{n+2}m_{\beta}}-\Lambda_{1}\right)\left(\Lambda_{a_{n+2}}-\Lambda_{1}\right)\\
    & = (\Lambda_{m_{\alpha}}-\Lambda_{1})(\Lambda_{a_{n+2}m_{\beta}}-\Lambda_{a_{n+2}}+\Lambda_{1})\\
    & = \gcd(m_{\alpha},a_{n+2}m_{\beta})\Lambda_{\operatorname{lcm}(m_{\alpha},a_{n+2}m_{\beta})}-\gcd(m_{\alpha},a_{n+2})\Lambda_{\operatorname{lcm}(m_{\alpha},a_{n+2})}+\Lambda_{m_{\alpha}}-\Lambda_{a_{n+2}m_{\beta}} \\
    & \quad +\Lambda_{a_{n+2}}-\Lambda_{1}.
\end{align*}
As $\gcd(m_{\alpha},a_{n+2}m_{\beta})=\gcd(m_{\alpha},a_{n+2})$ and $\operatorname{lcm}(m_{\alpha},a_{n+2}m_{\beta})=m_{\beta}\operatorname{lcm}(m_{\alpha},a_{n+2})$, we obtain
\begin{align*}
    \operatorname{div} \Delta & = \gcd(m_{\alpha},a_{n+2})\Lambda_{m_{\beta}\operatorname{lcm}(m_{\alpha},a_{n+2})}-\gcd(m_{\alpha},a_{n+2})\Lambda_{\operatorname{lcm}(m_{\alpha},a_{n+2})}+\Lambda_{m_{\alpha}}-\Lambda_{a_{n+2}m_{\beta}}\\
    & \quad +\Lambda_{a_{n+2}}-\Lambda_{1}.
\end{align*}
Therefore, the Alexander polynomial is
$$\Delta(t)=\dfrac{\left(t^{m_{\beta}\operatorname{lcm}(m_{\alpha},a_{n+2})}-1\right)^{\gcd(m_{\alpha},a_{n+2})}(t^{m_{\alpha}}-1)(t^{a_{n+2}}-1)}{\left(t^{\operatorname{lcm}(m_{\alpha},a_{n+2})}-1\right)^{\gcd(m_{\alpha},a_{n+2})}(t^{a_{n+2}m_{\beta}}-1)(t-1)}.$$
Of course, this implies that $\Delta(1)\neq0$ and hence the link $L_{f^{T}}$ is a rational homology $(2n+3)$-sphere.
Next, we compute the torsion. Here, we obtain $c_{\emptyset}=\gcd(m_{\alpha},a_{n+2})$,
$$ c_{n+1,n+2}=\dfrac{m_{\alpha}}{\gcd(m_{\alpha},a_{n+2})}, \quad c_{0,1,\dots,n}=\dfrac{a_{n+2}}{\gcd(m_{\alpha},a_{n+2})}, \quad c_{0,1,\dots,n,n+2}=m_{\beta}$$
and $c_{i_{1},\dots,i_{s}}=1$ in other cases. Then, we only need to compute $k_{\emptyset}, k_{n+1,n+2}, k_{0,1,\dots,n}$ and $k_{0,1,\dots,n, n+2}$. By definition, we have $k_{\emptyset}=1$, $k_{0,1,\dots,n}=0$ and 
$$k_{n+1,n+2}=1-\dfrac{1}{v_{n+1}}-\dfrac{1}{v_{n+2}}+\dfrac{u_{n+1}u_{n+2}}{v_{n+1}v_{n+2}\operatorname{lcm}(u_{n+1},u_{n+2})}=1.$$
For $k_{0,1,\dots,n,n+2}$, one can express this value as
\begin{equation}\label{eq:4.15}
    k_{0,1,\dots,n,n+2}=1-S_{1}+S_{2}-S_{3}+\dots -S_{n+1}+S_{n+2}
\end{equation}
where these $S_{j}$ are
\begin{align*}
    S_{1} & = \sum_{\substack{i=0,\\i\neq n+1}}^{n+2}\dfrac{1}{v_{i}}=\sum_{i=0}^{n}\dfrac{1}{\tilde{v}_{i}}+1, \\
    S_{2} & = \sum_{\substack{\tiny 0\leq i_{1}<i_{2}\leq n+2,\\i_{1},i_{2}\neq n+1}}\dfrac{u_{i_{1}}u_{i_{2}}}{v_{i_{1}}v_{i_{2}}\operatorname{lcm}(u_{i_{1}},u_{i_{2}})} = \sum_{\tiny 0\leq i_{1}<i_{2}\leq n}\dfrac{u_{i_{1}}u_{i_{2}}}{v_{i_{1}}v_{i_{2}}\operatorname{lcm}(u_{i_{1}},u_{i_{2}})}+ \sum_{i_{1}=0}^{n}\dfrac{u_{i_{1}}u_{n+2}}{v_{i_{1}}v_{n+2}\operatorname{lcm}(u_{i_{1}},u_{n+2})}, \\
    \vdots & = \quad \vdots \\
    S_{n+1} & = \sum_{\substack{\tiny 0\leq i_{1}<\dots <i_{n+1}\leq n+2,\\i_{1},\dots,i_{n+1}\neq n+1}}\dfrac{u_{i_{1}}\dots u_{i_{n+1}}}{v_{i_{1}}\dots v_{i_{n+1}}\operatorname{lcm}(u_{i_{1}},\dots, u_{i_{n+1}})} = \dfrac{u_{0}\dots u_{n}}{v_{0}\dots v_{n}\operatorname{lcm}(u_{0},\dots, u_{n})}\\
    & \quad + \sum_{0\leq i_{1}<\dots< i_{n}\leq n}\dfrac{u_{i_{1}}\dots u_{i_{n}}u_{n+2}}{v_{i_{1}}\dots v_{i_{n}}v_{n+2}\operatorname{lcm}(u_{i_{1}},\dots,u_{i_{n}},u_{n+2})},\\ \\
    S_{n+2} & = \dfrac{u_{0}\dots u_{n}u_{n+2}}{v_{0}\dots v_{n}v_{n+2}\operatorname{lcm}(u_{0},\dots,u_{n},u_{n+2})}.
\end{align*}

Simplifying for $j=2,3,\dots, n+1$, we obtain
$$S_{j}=m_{\alpha}^{j-1}\sum_{\tiny 0\leq i_{1}<\dots< i_{j}\leq n}\left(\dfrac{1}{\tilde{v}_{i_{1}}\dots\tilde{v}_{i_{j}}}\right)+\dfrac{m_{\alpha}^{j-1}a_{n+2}}{\operatorname{lcm}(m_{\alpha},a_{n+2})}\sum_{\tiny0\leq i_{1}<\dots< i_{j-1}\leq n}\left(\dfrac{1}{\tilde{v}_{i_{1}}\dots\tilde{v}_{i_{j-1}}}\right)$$
and $$S_{n+2}=\dfrac{m_{\alpha}^{n+1}a_{n+2}}{\operatorname{lcm}(m_{\alpha},a_{n+2})}\left(\dfrac{1}{\tilde{v}_{0}\dots\tilde{v}_{n}}\right).$$
Replacing these values in (\ref{eq:4.15}), we get
\begin{equation*}
    k_{0,1,\dots,n,n+2}=(\gcd(m_{\alpha},a_{n+2})-1)\left(\sum_{i=0}^{n}\dfrac{1}{\tilde{v}_{i}}-\sum_{0\leq i_{1}<i_{2}\leq n}\dfrac{m_{\alpha}}{\tilde{v}_{i_{1}}\tilde{v}_{i_{2}}}+\dots+\dfrac{m_{\alpha}^{n}}{\tilde{v}_{0}\dots\tilde{v}_{n}}\right).
\end{equation*}
Here we notice that the expression
$$\sum_{i=0}^{n}\dfrac{1}{\tilde{v}_{i}}-\sum_{0\leq i_{1}<i_{2}\leq n}\dfrac{m_{\alpha}}{\tilde{v}_{i_{1}}\tilde{v}_{i_{2}}}+\dots-\sum_{0\leq i_{1}<\dots<i_{n}\leq n}\dfrac{m_{\alpha}^{n-1}}{\tilde{v}_{i_{1}}\dots \tilde{v}_{i_{n}}}+\dfrac{m_{\alpha}^{n}}{\tilde{v}_{0}\dots\tilde{v}_{n}}$$
is simplified to
$$\dfrac{\prod_{i=0}^{n}(m_{\alpha}-\tilde{v}_{i})+\prod_{i=0}^{n}\tilde{v}_{i}}{m_{\alpha}\prod_{i=0}^{n}\tilde{v}_{i}}=\dfrac{\prod_{i=0}^{n}(a_{i}\tilde{v}_{i})+\prod_{i=0}^{n}\tilde{v}_{i}}{m_{\alpha}\prod_{i=0}^{n}\tilde{v}_{i}}=1.$$
This implies $k_{0,1,\dots,n,n+2}=\gcd(m_{\alpha},a_{n+2})-1$.
Finally, applying  Orlik's formula, we obtain
\begin{itemize}
    \item If $\gcd(m_{\alpha},a_{n+2})=1$, then $H_{n+1}(L_{f^{T}},\mathbb{Z})=\mathbb{Z}_{m_{\alpha}}$.
    \item If $\gcd(m_{\alpha},a_{n+2})\geq 2$, we have
    $$H_{n+1}(L_{f^{T}},\mathbb{Z})=\mathbb{Z}_{d}\oplus \underbrace{\mathbb{Z}_{m_{\beta}}\oplus\cdots\oplus\mathbb{Z}_{m_{\beta}}}_{(\gcd(m_{\alpha},a_{n+2})-2)-times}.$$
\end{itemize}
\end{proof}

\begin{rem}
It is interesting to notice that the condition $\operatorname{gcd}\left(m_\beta, a_{n+2}-1\right)=1$ in the theorem above can be reformulated in terms of  the Fano index $I$ as  $\operatorname{gcd}\left(m_\beta, I\right)=1.$  Indeed, since we are dealing with a 
 cycle-chain type polynomial 
$$f=f_0+z_{n+1}^{a_{n+1}}+z_{n+1} z_{n+2}^{a_{n+2}}$$  with 
$\mathbf{w}=\left(w_0, w_1, \ldots, w_n, w_{n+1}, w_{n+2}\right)=\left(m_\beta \mathbf{v}, m_\alpha v_{n+1}, m_\alpha v_{n+2}\right)$  and  $d=m_\alpha m_\beta,$ satisfying the conditions in (\ref{eq:4.5}), using identical arguments as in the proof of Lemma \ref{lem:4.1.3}, we obtain  that $v_{n+1}=1.$ Then from $w_{n+1}+a_{n+2}w_{n+2}=d$ it follows that 
$$1+a_{n+2}v_{n+2}=m_{\beta}.$$
From this equality (after multiplying and dividing by $m_{\alpha}$) one has  
$$a_{n+2}-1=\frac{m_{\beta}-1-v_{n+2}}{v_{n+2}}=\frac{d-w_{n+1}-w_{n+2}}{w_{n+2}}.$$ 
Since  the Fano index is given by $I=\displaystyle{\sum_{i=0}^{n+2}w_i-d,}$ we obtain 
$$a_{n+2}-1=\frac{w_0+\ldots +w_n-I}{w_{n+2}}=\frac{m_{\beta}(v_0+\ldots +v_n)-I}{w_{n+2}}.$$ 
Then $$\operatorname{gcd}\left(m_\beta, a_{n+2}-1\right)=\operatorname{gcd}\left(m_\beta, \frac{m_{\beta}(v_0+\ldots +v_n)-I}{w_{n+2}}\right).$$
So $\operatorname{gcd}\left(m_\beta, a_{n+2}-1\right)=1$ if and only $\operatorname{gcd}\left(m_\beta, I\right)=1.$ Notice that in case the Fano index equals 1, this condition is satisfied automatically. 
\end{rem}

\medskip

In case $\operatorname{gcd}\left(m_\beta, a_{n+2}-1\right)=r>1$ we still can construct rational homology spheres in polynomials of type II: 

\begin{corollary}\label{cor:5.4.4} Let us consider $f$ of type II satisfying all the hypotheses in Theorem \ref{prop:4.2.6} except that this time  $\operatorname{gcd}\left(m_\beta, a_{n+2}-1\right)=r>1.$ If 
$$\operatorname{gcd}\left(m_{\alpha}, a_{n+2}\right)=1,$$  the corresponding dual link $L_f^T$ is a rational homology sphere.
\end{corollary}

\begin{proof}
Indeed, from the proof of the previous theorem it is not difficult see that  the $u_i's$ and the $v_i's$ are unchanged unless $i=n+1,$ where we have:   $$u_{n+1}=\dfrac{\tilde{d}}{\gcd(\tilde{d},\tilde{w}_{n+1})}=\dfrac{a_{n+2}m_{\alpha}m_{\beta}}{\gcd(a_{n+2}m_{\alpha}m_{\beta},m_{\alpha}(a_{n+2}-1))}=a_{n+2}\frac{m_{\beta}}{r}$$
and 
$$v_{n+1}=\dfrac{\tilde{w}_{n+1}}{\gcd(\tilde{d},\tilde{w}_{n+1})}=\dfrac{m_{\alpha}(a_{n+2}-1)}{\gcd(a_{n+2}m_{\alpha}m_{\beta},m_{\alpha}(a_{n+2}-1))}=\frac{a_{n+2}-1}{r}.$$
From there, we compute the Alexander polynomial as follows 
\begin{align*}
    \operatorname{div} \Delta 
    & = (\Lambda_{m_{\alpha}}-\Lambda_{1})(r\Lambda_{a_{n+2}\frac{m_{\beta}}{r}}-\Lambda_{a_{n+2}}+\Lambda_{1})\\
    & = r\gcd(m_{\alpha},a_{n+2})\Lambda_{\operatorname{lcm}(m_{\alpha},a_{n+2}\frac{m_{\beta}}{r})}-\gcd(m_{\alpha},a_{n+2})\Lambda_{\operatorname{lcm}(m_{\alpha},a_{n+2})}+\Lambda_{m_{\alpha}}-r\Lambda_{a_{n+2}\frac{m_{\beta}}{r}} \\
    & \quad +\Lambda_{a_{n+2}}-\Lambda_{1}.
\end{align*}
From $\gcd(m_{\alpha},a_{n+2})=1$ it follows that  the Alexander polynomial is
$$\Delta(t)=\dfrac{\left(t^{m_{\alpha}a_{n+2}\frac{m_\beta}{r}}-1\right)^{r}(t^{m_{\alpha}}-1)(t^{a_{n+2}}-1)}
{\left(t^{m_{\alpha}a_{n+2}}-1\right)(t^{a_{n+2}\frac{m_{\beta}}{r}}-1)^r(t-1)}.$$
From here, $|\Delta(1)|=m_\alpha.$ Thus the link $L_f^T$ is a rational homology sphere.
\end{proof}

\begin{rem}  We have the following remarks for the the well-formedness of the hypersurface. Recall that the weight vector associated to $X_{f^T}$ is given by 
$$\mathbf{\tilde{w}}=(\tilde{w}_{0},\dots,\tilde{w}_{n+1}, \tilde{w}_{n+2})=(a_{n+2}m_{\beta}\tilde{v}_{0},a_{n+2}m_{\beta}\tilde{v}_{1},\dots, a_{n+2}m_{\beta}\tilde{v}_{n},m_{\alpha}(a_{n+2}-1),m_{\alpha}m_{\beta})$$ 
\begin{itemize} 
\item If $f$ is as in Theorem \ref{prop:4.2.6}, then the resulting hypersurface via $BH$-transpose rule $X_{f^T}$ is not well-formed.  clearly 
the first $n$ weights have as common factor $a_{n+2}m_{\beta}.$ Notice that 
$$\gcd(\tilde{w}_{n+1}, m_\beta)=\gcd(m_{\alpha}(a_{n+2}-1), m_{\beta})=1.$$
Then  any solution vector $(\tilde{w}_{0},\dots,\tilde{w}_{n+1}, \tilde{w}_{n+2})$ with integer entries always contains $m_{\beta}$ in all but one weight. Thus, from Lemma \ref{lem:2.1}, $X_{f^T}$ is not well-formed.
 \item  If $f$ is as in Corollary \ref{cor:5.4.4}, that is $\operatorname{gcd}\left(m_\beta, a_{n+2}-1\right)=r>1$ due to the condition  
$\operatorname{gcd}\left(m_{\alpha}, a_{n+2}\right)=1,$ one also concludes that $X_{f^T}$ is not well-formed.
 \end{itemize}
\end{rem}

%$\operatorname{gcd}\left(m_\alpha, m_\beta\right)=1$ and  $\operatorname{gcd}\left(m_\beta, v_{n+1}\right)=\operatorname{gcd}\left(m_\beta, v_{n+2}\right)=1.$

\section{Sasakian-Einstein structure on links of cycle polynomials}

We are interested in links that admit a Sasakian-Einstein structure and where the B-H rule preserves this feature. In \cite{CGL}, we work in dimension 7. In this section, we improve some results of this paper and extend them to higher dimensions. 

%For positive Sasakian structures, we have the following result given by Gomez in \cite{Gom}.
%\begin{lem}\label{lem:4.3.1}
%    Let $f$ be an invertible polynomial and $X_{f}\subset\mathbb{P}(\mathbf{w})$ be its corresponding weighted hypersurface. We denote by $f^T$ the polynomial obtained by applying the B-H rule to $f$ and $X_{f^T}$ its respective weighted hypersurface. Then if $X_{f}$ is a Fano hypersurface, then so is $X_{f^{T}}$. Moreover, this implies that if the link $L_{f}$ admits a positive Ricci curvature Sasaki metric, then the link $L_{f^{T}}$ associated to the polynomial $f^{T}$ also admits a Sasaki metric of positive Ricci curvature.
%\end{lem}
%
We have the following two technical lemmas:
\begin{lem}\label{lem:4.3.2}
    Let $a_{0},a_{1}$ and $a_{2}$ be integer numbers greater than or equal to 2. If we denote
    $$v_{0}=1-a_{1}+a_{1}a_{2}, \ \ v_{1}=1-a_{2}+a_{2}a_{0} \ \ \mbox{ and } \ \ v_{2}=1-a_{0}+a_{0}a_{1},$$
    then 
    \begin{equation}\label{eq:4.16}
        1+a_{0}a_{1}a_{2}<\dfrac{4}{3}\min_{i,j}\{ v_{i}v_{j}\}.
    \end{equation}
\end{lem}

\begin{proof}

    We can suppose without loss of generality that there are two cases: $v_{0}\leq v_{1}\leq v_{2}$ or $v_{0}\leq v_{2} \leq v_{1}$.
    
    \begin{itemize}
        \item If $v_{0}\leq v_{1}\leq v_{2}$. Here $\min_{i,j}\{v_{i}v_{j}\}=v_{0}v_{1}$.
    Calculating, we obtain
    $$v_{0}v_{1}=1-a_{2}-a_{1}+a_{0}a_{2}+2a_{1}a_{2}-a_{0}a_{1}a_{2}-a_{1}a_{2}^2+a_{0}a_{1}a_{2}^2.$$
    We will show that $1+a_{0}a_{1}a_{2}\leq v_{0}v_{1}$. Notice that this is equivalent to
    \begin{equation}\label{eq:4.17}
        2a_{1}a_{2}(a_{0}-1)\leq a_{1}a_{2}^2(a_{0}-1)+a_{0}a_{2}-a_{2}-a_{1}.
    \end{equation}
    Since $v_{0}\leq v_{1}$, we have $1-a_{1}+a_{1}a_{2}\leq 1-a_{2}+a_{2}a_{0}$. This implies that
    $$a_{1}\leq a_{1}(a_{2}-1) \leq a_{2}(a_{0}-1).$$
    As $2a_{1}a_{2}(a_{0}-1)\leq a_{1}a_{2}^2(a_{0}-1)$ and $a_{1}\leq a_{0}a_{2}-a_{2}$, then (\ref{eq:4.17}) holds. Therefore, we have
    $$1+a_{0}a_{1}a_{2}\leq\min_{i,j}\{ v_{i}v_{j}\} <\dfrac{4}{3}\min_{i,j}\{ v_{i}v_{j}\}.$$
    \item If $v_{0}\leq v_{2}\leq v_{1}$. Here we have $\min_{i,j}\{v_{i}v_{j}\}=v_{0}v_{2}$.
    As 
    $$v_{0}v_{2}=1-a_{0}-a_{1}+2a_{0}a_{1}-a_{0}a_{1}^2+a_{1}a_{2}-a_{0}a_{1}a_{2}+a_{0}a_{1}^2a_{2},$$
    then the inequality (\ref{eq:4.16}) is equivalent to
    $$3+3a_{0}a_{1}a_{2}<4-4a_{0}-4a_{1}+8a_{0}a_{1}-4a_{0}a_{1}^2+4a_{1}a_{2}-4a_{0}a_{1}a_{2}+4a_{0}a_{1}^2a_{2}.$$
    Simplifying, we write the above expression as
    \begin{equation}\label{eq:4.18}
        8a_{0}a_{1}(a_{2}-1)< 4a_{0}a_{1}^2(a_{2}-1) + 4a_{1}(a_{2}-1)+a_{0}(a_{1}a_{2}-4)+1.
    \end{equation}
As $8a_{0}a_{1}(a_{2}-1)\leq 4a_{0}a_{1}^2(a_{2}-1)$ and $0\leq 4a_{1}(a_{2}-1)+a_{0}(a_{1}a_{2}-4)$, we find that equation (\ref{eq:4.18}) is true.
    \end{itemize}
\end{proof}
In general, when $n$ is even and greater than $2$, we have
\begin{lem}\label{lem:4.3.3}
    Given the integer numbers $a_{0}, a_{1},\dots, a_{n}$, where each $a_{i}\geq 2$ and  the integer number $n\geq4$ is even. If we define the number
    $v_{i}$ as in (\ref{eq:4.3}), i.e.
    \begin{equation*}
    v_{i}=1-a_{i+1}+a_{i+1}a_{i+2}-\dots +a_{i+1}a_{i+2}\dots a_{i+n-1}a_{i+n},
\end{equation*}
where the subscripts are taken$\mod (n+1)$, then
\begin{equation}\label{eq:4.19}
    1+a_{0}a_{1}\dots a_{n}<   \left(\dfrac{n+2}{n+1}\right)\min_{i,j}\{v_{i}v_{j}\}.
\end{equation}
\end{lem}
\begin{proof}
    We can assume without loss of generality that $i=0$. Therefore, let us see that
    \begin{equation}\label{eq:4.20}
    1+a_{0}a_{1}\dots a_{n}<   \left(\dfrac{n+2}{n+1}\right)\min_{j}\{v_{0}v_{j}\}.
\end{equation}
Multiplying $v_{0}$ and $v_{j}$, we obtain
\begin{equation*}
    v_{0}v_{j}=1-S_{1}+S_{2}-S_{3}+\cdots -S_{2n-1}+S_{2n},
\end{equation*}
where
\begin{align*}
    S_{1}& = a_{1}+a_{j+1},\\
    S_{2}& = a_{1}a_{2}+a_{1}a_{j+1}+a_{j+1}a_{j+2},\\
    S_{3} & = a_{1}a_{2}a_{3}+a_{1}a_{2}a_{j+1}+a_{1}a_{j+1}a_{j+2}+a_{j+1}a_{j+2}a_{j+3},\\
    \vdots & \quad  \vdots\\
    S_{n} & = a_{1}a_{2}\dots a_{n}+a_{1}a_{2}\dots a_{n-1}a_{j+1}+\dots+a_{1}a_{j+1}\dots a_{j-2}+a_{j+1}a_{j+2}\dots a_{j-1},\\
    \vdots & \quad \vdots \\
    S_{2n-1} & = a_{1}a_{2}\dots a_{n}a_{j+1}a_{j+2}\dots a_{j-2}+a_{1}a_{2}\dots a_{n-1}a_{j+1}a_{j+2}\dots a_{j-1},\\
    S_{2n} & = a_{1}a_{2}\dots a_{n}a_{j+1}a_{j+2}\dots a_{j-1}.
\end{align*}
From $S_{n+1}$, we extract the term $a_{0}a_{1}\dots a_{n}$. If we denote $\tilde{S}_{n+1}=S_{n+1}-a_{0}a_{1}\dots a_{n}$, then
\begin{align*}
    \tilde{S}_{n+1} & =a_{1}a_{2}\dots a_{n}a_{j+1}+\dots+a_{1}a_{2}\dots a_{j+1}a_{j+1}\dots a_{n}\\
    & \quad   +a_{1}a_{2}\dots a_{j-1}a_{j+1}\dots a_{0}a_{1}+\dots +a_{1}a_{j+1}\dots a_{j-1},
\end{align*}
which has $n-1$ summands. Then, we can write the inequality (\ref{eq:4.20}) as
\begin{equation}\label{eq:4.21}
    2(n+2)a_{0}a_{1}\dots a_{n}<1+a_{0}a_{1}\dots a_{n}+(n+2)\left[ -S_{1}+S_{2}-\dots +S_{n}-\tilde{S}_{n+1}+\dots-S_{2n-1}+S_{2n}\right].
\end{equation}
Now, we notice that 
\begin{equation}\label{eq:4.22}
    S_{n+k+1}\geq S_{n+k}, \quad \text{for }k\in\{3,5,\dots,n-1\}.
\end{equation}
Indeed, we have
\begin{align*}
    S_{n+k}& = a_{1}\dots a_{n}a_{j+1}\dots a_{j+k}+a_{1}\dots a_{n-1}a_{j+1}\dots a_{j+k+1} +\\
    & \quad \dots +a_{1}\dots a_{k+1}a_{j+1}\dots a_{j-2}+a_{1}\dots a_{k}a_{j+1}\dots a_{j-1},
\end{align*}
which has $n-k+1$ summands, and
\begin{align*}
    S_{n+k+1}& = a_{1}\dots a_{n}a_{j+1}\dots a_{j+k+1}+a_{1}\dots a_{n-1}a_{j+1}\dots a_{j+k+2} +\\
    & \hspace{1cm} \dots +a_{1}\dots a_{k+2}a_{j+1}\dots a_{j-2}+a_{1}\dots a_{k+1}a_{j+1}\dots a_{j-1},
\end{align*}
which has $(n-k)$ summands. Clearly, the sum of the first $(n-k-1)$ summands of $S_{n+k+1}$ is greater than the sum of the first $(n-k-1)$ summands of $S_{n+k}$. Also, the last term in $S_{n+k+1}$ is at least the sum of the last two terms in $s_{n+k}$:
\begin{align*}
    & a_{1}\dots a_{k}a_{k+1}a_{j+1}\dots a_{j-1}-a_{1}\dots a_{k+2}a_{j+1}\dots a_{j-2}-a_{1}\dots a_{k+1}a_{j+1}\dots a_{j-1} \\
    & = a_{1}\dots a_{k}a_{j+1}\dots a_{j-2}(a_{k+1}a_{j-1}-a_{k+1}-a_{j-1})\\
    & \geq 0.
\end{align*}
Thus, in (\ref{eq:4.21}) we only need to prove that
\begin{equation}\label{eq:4.23}
    2(n+2)a_{0}a_{1}\dots a_{n}<1+a_{0}a_{1}\dots a_{n}+(n+2)\left[ -S_{1}+S_{2}-\dots +S_{n}-\tilde{S}_{n+1}+S_{n+2}\right].
\end{equation}
Now, we can find two terms of the form $a_{0}a_{1}\dots\hat{a}_{j}\dots a_{n}$ in the expression of $S_{n}$. We extract these two terms from $S_{n}$ and write
\begin{align*}
    \tilde{\tilde{S}}_{n} & =S_{n}-2a_{0}a_{1}\dots\hat{a}_{j}\dots a_{n}\\
    & = a_{1}a_{2}\dots a_{n}+\dots +a_{1}a_{2}\dots a_{j}a_{j+1}\dots a_{n}+a_{1}a_{2}\dots a_{j-2}a_{j+1}\dots a_{0}a_{1}+\dots+a_{1}a_{j+1}\dots a_{j-2}. 
\end{align*}
Then inequality (\ref{eq:4.23}) is equivalent to
\begin{equation}\label{eq:4.24}
    2(n+2)a_{0}a_{1}\dots \hat{a}_{j}\dots a_{n}(a_{j}-1)<1+a_{0}a_{1}\dots a_{n}+(n+2)\left[ -S_{1}+S_{2}-\dots +\tilde{\tilde{S}}_{n}-\tilde{S}_{n+1}+S_{n+2}\right].
\end{equation}
On the other hand, from $S_{n+2}$, we can extract the summand $a_{1}a_{2}\dots a_{j}a_{j+1}a_{j+2}\dots a_{n}a_{0}a_{1}$. Then, if we write $\tilde{S}_{n+2}=S_{n+2}-a_{1}a_{2}\dots a_{j}a_{j+1}a_{j+2}\dots a_{n}a_{0}a_{1}$, we obtain
\begin{align*}
    \tilde{S}_{n+2} & = a_{1}\dots a_{n}a_{j+1}a_{j+2}+a_{1}\dots a_{n-1}a_{j+1}a_{j+2}a_{j+3}+\dots +a_{1}a_{2}\dots a_{j+1}a_{j+1}a_{j+2}\dots a_{n}a_{0}\\
    & \quad + a_{1}a_{2}\dots a_{j-1}a_{j+1}a_{j+2}\dots a_{n}a_{0}a_{1}a_{2}+\dots+a_{1}a_{2}a_{3}a_{j+1}\dots a_{j-2}+a_{1}a_{2}a_{j+1}\dots a_{j-1},
\end{align*}
which has $(n-2)$ summands. Thus, inequality (\ref{eq:4.24}) is equivalent to
\begin{align}
    2(n+2)a_{0}a_{1}\dots \hat{a}_{j}\dots a_{n}(a_{j}-1)& <1+a_{0}a_{1}\dots a_{n}\\
    & \quad+(n+2)\left[ -S_{1}+S_{2}-\dots +\tilde{\tilde{S}}_{n}-\tilde{S}_{n+1}+\tilde{S}_{n+2}\right] \notag \\
    & \quad +(n+2)a_{1}a_{2}\dots a_{j-1}a_{j}a_{j+1}a_{j+2}\dots a_{n}a_{0}a_{1}. \label{eq:4.25}
\end{align}
We notice that $\tilde{S}_{n+2}>\tilde{S}_{n+1}$. Indeed, we see that the first summand of $\tilde{S}_{n+2}$ is at least the sum of the first two summands of $\tilde{S}_{n+1}$. Moreover, the sum of the other terms of $\tilde{S}_{n+2}$ is greater than the sum of remaining summands of $\tilde{S}_{n+1}$. In a similar way, we have $\tilde{\tilde{S}}_{n}\geq S_{n-1}$. Also, it is evident that $S_{2}>S_{1}$, $S_{4}>S_{3}, \dots, S_{n-2}>S_{n-3}$. Thus the sum in the brackets in (\ref{eq:4.25}) is positive. Finally, since $a_{1}\geq2$, we have
$$ 2(n+2)a_{0}a_{1}\dots \hat{a}_{j}\dots a_{n}(a_{j}-1)<(n+2)a_{1}a_{2}\dots a_{j-1}a_{j}a_{j+1}a_{j+2}\dots a_{n}a_{0}a_{1}.$$
We conclude that the inequality in (\ref{eq:4.25}) is valid.
\end{proof}
From Lemma \ref{lem:4.3.3}, we obtain the following result for cycle polynomials.
\begin{theorem}\label{prop:4.3.4}
    Let $f_{0}$ be a polynomial as in (\ref{eq:4.1}) of degree $m_{\alpha}$ and weight vector $\mathbf{v}=(v_{0},v_{1},\dots, v_{n})$, with $n\geq4$, such that $I=|\mathbf{v}|-m_{\alpha}=1$. Then $f_{0}$ satisfies the condition $K1$ and the weighted hypersurface $X_{f_{0}}\subset\mathbb{P}(\mathbf{w})$ admits a Kähler-Einstein structure. Moreover, its corresponding link $L_{f_{0}}$ is a rational homology sphere that admits a Sasaki-Einstein metric.    
\end{theorem}
\begin{proof}
    We consider the cycle polynomial $f_{0}=z_{n}z_{0}^{a_{0}}+z_{0}z_{1}^{a_{1}}+\dots+z_{n-1}z_{n}^{a_{n}}$ with weight vector $\mathbf{v}=(v_{0},\dots,v_{n})$ and degree $m_{\alpha}$, such that $|\mathbf{v}|-m_{\alpha}=1$. If $v^{*}=\gcd(v_{0},v_{1},\dots,v_{n})$, then we will prove that $v^{*}=1$. Since the weights $v_{i}$'s and the degree $m_{\alpha}$ verify the equations in (\ref{eq:4.2}):
    $$v_{n}+a_{0}v_{0}=m_{\alpha}, \quad v_{0}+a_{1}v_{1}=m_{\alpha}, \quad \dots \quad , v_{n-1}+a_{n}v_{n}=m_{\alpha},$$
     we have $v^{*}|m_{\alpha}$. As $|\mathbf{v}|-m_{\alpha}=1$ and $v^{*}|(|\mathbf{v}|-m_{\alpha})$, we conclude that $v^{*}=1$. This means that $f_{0}$ verifies the condition $K1$, which implies that $L_{f_{0}}$ is a rational homology sphere. Moreover, from Lemma \ref{lem:4.3.3}, we have
    $$m_{\alpha}<\dfrac{n+2}{n+1}\min_{i,j}\{v_{i}v_{j}\}<\dfrac{n}{n-1}\min_{i,j}\{v_{i}v_{j}\}.$$
    As the index $I=|\mathbf{v}|-m_{\alpha}=1$, from Theorem \ref{prop:1.5.8} we see that $X_{f_{0}}$ admits a Kähler-Einstein orbifold metric and  $L_{f_{0}}$ has a positive Sasakian structure and admits a Sasaki-Einstein metric.
\end{proof}
As a consequence of the above proposition, we see that the link $L_{f_{0}^{T}}$ of the dual polynomial $f_{0}^{T}$ also admits a Sasaki-Einstein metric. 

\begin{corollary}\label{cor:4.3.5}
    Let $f_{0}$ be a polynomial as in (\ref{eq:4.1}) of degree $m_{\alpha}$ and weight vector $\mathbf{v}=(v_{0},v_{1},\dots, v_{n})$, with $n\geq4$, such that $I=|\mathbf{v}|-m_{\alpha}=1$. If $f_{0}^{T}$ is the dual polynomial of $f_{0}$ obtained by the B-H rule, then $f_{0}^{T}$ also satisfies the condition $K1$, its weighted hypersurface $X_{f_{0}^{T}}$ admits a Kähler-Einstein structure and its link $L_{f_{0}^{T}}$ is a rational homology sphere which admits a Sasaki-Einstein metric.
\end{corollary}

\begin{proof}
    Let $f_{0}^{T}$ be the dual polynomial of $f_{0}$ obtained by B-H rule. From Theorem \ref{prop:4.2.2}, we know that $\tilde{m}_{\alpha}=1+a_{0}\dots a_{n}$ and the vector $\tilde{\mathbf{v}}=(\tilde{v}_{0},\tilde{v}_{1},\dots,\tilde{v}_{n})$ satisfy the equations:
    $$a_{0}\tilde{v}_{0}+\tilde{v}_{1}=\tilde{m}_{\alpha}, \quad a_{1}\tilde{v}_{1}+\tilde{v}_{2}=\tilde{m}_{\alpha}, \quad \dots , a_{n}\tilde{v}_{n}+\tilde{v}_{0}=\tilde{m}_{\alpha},$$
    that define the monomials in $f_{0}^{T}$, where each $\tilde{v}_{i}$ is defined as in (\ref{eq:4.11}):
    \begin{equation*}
        \tilde{v}_{i}=1-a_{i-1}+a_{i-1}a_{i-2}-\dots+a_{i-1}a_{i-2}\cdots a_{i-n}, 
    \end{equation*}
    where the subscripts are taken$\mod (n+1)$.  Let us see that $\gcd(\tilde{v}_{0},\tilde{v}_{1},\dots,\tilde{v}_{n})=1$.
    If we denote $\tilde{v}^{*}=\gcd(\tilde{v}_{0},\dots,\tilde{v}_{n})$, then $\tilde{v}^{*}|m_{\alpha}$. From Lemma \ref{lem:4.2.1}, we know $|\mathbf{v}|=|\mathbf{\tilde{v}}|$, which implies that $|\mathbf{\tilde{v}}|-\tilde{m}_{\alpha}=|\mathbf{v}|-m_{\alpha}=1$. Then $\tilde{v}^{*}=1$. Following the same argument as above, we have that $X_{f_{0}^{T}}$ admits a Kähler-Einstein orbifold metric and the link $L_{f_{0}^{T}}$ has a positive Sasakian structure and carries a Sasaki-Einstein structure.
\end{proof}
Let us see some examples.
\begin{exm}\label{ex:4.3.6}
    We consider the weight vector $\mathbf{w}=(3073, 712, 2211, 151, 1199)$ found in the sporadic list of anticanonically embedded quasi-smooth Fano hypersurfaces in weighted projective 4-space given by Johnson and Kollar in \cite{JK}. Here, the hypersurface $X_{f}\subset\mathbb{P}(\mathbf{w})$ is defined by the cycle polynomial
    $$f_{0}=z_{4}z_{0}^{2}+z_{0}z_{1}^{6}+z_{1}z_{2}^{3}+z_{2}z_{3}^{34}+z_{3}z_{4}^{6}.$$
    of degree $d=7345$. Since $I=1$, by Theorem \ref{prop:4.3.4}, we have that the link $L_{f_{0}}$ associated to $f_{0}$ is a rational homology $7$-sphere which admits a Sasaki-Einstein structure. Moreover,  if we consider its dual polynomial
    $$f_{0}^{T}=z_{0}^{2}z_{1}+z_{1}^{6}z_{2}+z_{2}^{3}z_{3}+z_{3}^{34}z_{4}+z_{4}^{6}z_{0},$$
    by the Corollary \ref{cor:4.3.5} we have that its link $L_{f_{0}^{T}}$ is a rational homology $7$-sphere which admits also a Sasaki-Einstein structure. In addition, by Proposition \ref{prop:4.1.2}, we have that
    $$H_{3}(L_{f_{0}},\mathbb{Z})=H_{3}(L_{f_{0}^{T}},\mathbb{Z})=\mathbb{Z}_{7345}.$$
\end{exm}

\begin{exm}\label{ex:4.3.7}
    We consider the hypersurface $X_{f}\subset\mathbb{P}(\mathbf{w})$ defined by the cycle polynomial
    $$f_{0}=z_{6}z_{0}^{2}+z_{0}z_{1}^{3}+z_{1}z_{2}^{24}+z_{2}z_{3}^{25}+z_{3}z_{4}^{12}+z_{4}z_{5}^{8}+z_{5}z_{6}^{9}.$$
    Computing its weights $w_{i}$ and degree $d$, we obtain 
    $$\mathbf{w}=(1402270,569377,105876,120181,249185,357652,305861)$$
     and $d=3110401$. Clearly, we notice that $I=|\mathbf{w}|-d=1$. Then by Proposition \ref{prop:4.3.4}, we have that the link $L_{f_{0}}$ is a rational homology $11$-sphere which admits a Sasaki-Einstein metric. Moreover, using the Corollary \ref{cor:4.3.5}, if we consider its dual polynomial:
     $$f_{0}^{T}=z_{0}^{2}z_{1}+z_{1}^{3}z_{2}+z_{2}^{24}z_{3}+z_{3}^{25}z_{4}+z_{4}^{12}z_{5}+z_{5}^{8}z_{6}+z_{6}^{9}z_{0},$$
     we have that its link $L_{f_{0}^{T}}$ is also a rational homology $11$-sphere and admits a Sasaki-Einstein structure. Finally, using the Proposition \ref{prop:4.1.2}, we have that
     $$H_{5}(L_{f_{0}},\mathbb{Z})=H_{5}(L_{f_{0}^{T}},\mathbb{Z})=\mathbb{Z}_{3110401}.$$
\end{exm}

\begin{exm}\label{ex:4.3.8}
    Given the weighted hypersurface $X_{f_{0}}\subset\mathbb{P}(\mathbf{w})$ which is defined by the cycle polynomial
    $$f_{0}=z_{8}z_{0}^{4}+z_{0}z_{1}^{11}+z_{1}z_{2}^{14}+z_{2}z_{3}^{10}+z_{3}z_{4}^{8}+z_{4}z_{5}^{13}+z_{5}z_{6}^{4}+z_{6}z_{7}^{15}+z_{7}z_{8}^{13},$$
    we find its weights $\mathbf{w}=(w_{0},w_{1},\dots,w_{8})$ and degree $d$ as
    \begin{align*}
        \mathbf{w} = ( & 115806604, 34899327, 33199991, 46649921, 56631160, 34082157, 116404261, \\
        & 25552996,36472785)
    \end{align*}
    and $d=499699201$. Here, we notice that $I=|\mathbf{w}|-d=1$. Using Theorem \ref{prop:4.3.4}, we see that the link $L_{f_{0}}$ is a rational homology $15$-sphere that admits a Sasaki-Einstein metric. Moreover, if we consider its dual polynomial
    $$f_{0}^{T}=z_{0}^{4}z_{1}+z_{1}^{11}z_{2}+z_{2}^{14}z_{3}+z_{3}^{10}z_{4}+z_{4}^{8}z_{5}+z_{5}^{13}z_{6}+z_{6}^{4}z_{7}+z_{7}^{15}z_{8}+z_{8}^{13}z_{0},$$
    by Corollary \ref{cor:4.3.5} we also have that $L_{f_{0}^{T}}$ is a rational homology $15$-sphere and admits a Sasaki-Einstein metric. Then, by Proposition \ref{prop:4.1.2} we obtain
    $$H_{7}(L_{f_{0}},\mathbb{Z})=H_{7}(L_{f_{0}^{T}},\mathbb{Z})=\mathbb{Z}_{499699201}.$$
\end{exm}

\section{Sasakian-Einstein structure on links of hypersurfaces of Thom-Sebastiani sums}

In this section, we prove some results about the existence of Sasaki-Einstein metrics on links of hypersurfaces via rational subvarieties of codimension two, more precisely, on  links that arise from polynomials of type I, II or III which are defined in Section 4.1. We remember that these polynomials consist in Thom-Sebastiani sums constructed from a cycle polynomial $f_{0}=z_{n}z_{0}^{a_{0}}+z_{0}z_{1}^{a_{1}}+\dots +z_{n-1}z_{n}^{a_{n}}$. 

Let us begin with the next proposition, where we prove that for polynomials of type I or III, the Sasaki-Einstein structure of its link is invariant by the B-H rule. More precisely, we will see that inequality (\ref{eq:3.1}) is preserved.

\begin{theo}\label{prop:4.4.1}
    Let $$f=f_{0}+g(z_{n+1},z_{n+2})$$ be a polynomial of type I or III of degree $d=m_{\alpha}m_{\beta}$ and weight vector $\mathbf{w}=(m_{\beta}\mathbf{v},m_{\alpha}v_{n+1},m_{\alpha}v_{n+2})$ which verify the conditions (\ref{eq:4.5}), such that $I=|\mathbf{w}|-d=1$ or $I=2$. We denote by $f^T$ the polynomial obtained by the B-H rule. Then the links  $L_{f}$ and $L_{f^{T}}$ are rational homology spheres. Moreover, if the weight vector $\mathbf{w}$ and the degree $d$ of $f$ verify the inequality (\ref{eq:3.1}), which implies that $L_{f}$ admits a Sasaki-Einstein metric, then $L_{f^T}$ also admits a Sasaki-Einstein metric.
\end{theo}
\begin{proof}
    By Theorem \ref{prop:4.2.4}, we have  $f^{T}=f_{0}^{T}+g(z_{n+1},z_{n+2})$. Let us see that $f_{0}$ and $f_{0}^{T}$ satisfy the condition $K1$. We write $v^{*}=\gcd(v_{0},\dots,v_{n})$, which implies that $v^{*}|m_{\alpha}$. First, we suppose that $I=1$. As 
    $$I=|\mathbf{w}|-d=m_{\beta}|\mathbf{v}|+m_{\alpha}v_{n+1}+m_{\alpha}v_{n+2}-m_{\alpha}m_{\beta}=1,$$
    then we have $v^{*}=1$. In the other case, when $I=2$, we have
$$I=|\mathbf{w}|-d=m_{\beta}|\mathbf{v}|+m_{\alpha}v_{n+1}+m_{\alpha}v_{n+2}-m_{\alpha}m_{\beta}=2.$$
This implies that $v^{*}\mid 2$. If $v^{*}=2$, then $m_{\alpha}=1+a_{0}\dots a_{n}$ is even. Thus, we obtain that $a_{0},a_{1},\dots, a_{n}$ are odd. As a consequence, from Formula (\ref{eq:4.3}), we have that $v_{0},v_{1},\dots, v_{n}$ are odd, which is a contradiction. Therefore, in any case, $v^{*}=1$.

Now, we consider the weight vector $\mathbf{\tilde{w}}=(m_{\beta}\mathbf{\tilde{v}},m_{\alpha}v_{n+1},m_{\alpha}v_{n+2})$ and degree $\tilde{d}=m_{\alpha}m_{\beta}$ of polynomial $f^{T}$, where $\mathbf{\tilde{v}}=(\tilde{v}_{0},\dots,\tilde{v}_{n})$ is defined as in (\ref{eq:4.11}). We denote $\tilde{v}^{*}=\gcd(\tilde{v}_{0},\dots,\tilde{v}_{n})$. Let us see that $\tilde{v}^{*}=1$. By equation (\ref{eq:4.10}), we have $\tilde{v}^{*}|m_{\alpha}$. Moreover, by Lemma \ref{lem:4.2.1} we have $|\mathbf{\tilde{v}}|=|\mathbf{v}|$, which implies that $\tilde{I}=|\mathbf{\tilde{w}}|-d=|\mathbf{\tilde{w}}|-d=I$. Following a similar process as above, we conclude that $\tilde{v}^{*}=1$. As $f_{0}$ and $f_{0}^{T}$ verify the condition $K1$, then by Theorem \ref{prop:4.2.4} we have that $L_{f}$ and $L_{f^{T}}$ are rational homology spheres. 

    Next, we assume that $\mathbf{w}$ and $d$ verify inequality (\ref{eq:3.1}):
    $$Id<\dfrac{n+2}{n+1}\min_{i,j}\{w_{i}w_{j}\}$$
    and then $L_{f}$ admits a Sasaki-Einstein metric. Next, we prove that $\mathbf{\tilde{w}}$ and $\tilde{d}=d$ also satisfy (\ref{eq:3.1}). We write $\tilde{w}_{i_{0}}\tilde{w}_{j_{0}}=\min_{i,j}\{\tilde{w}_{i}\tilde{w}_{j}\}$. Here, we have the following cases:
    \begin{itemize}
        \item If $i_{0}\in\{0,\dots,n\}$ and $j_{0}=n+1$ (for $j_{0}=n+2$, it is similar), we have $\tilde{w}_{i_{0}}\tilde{w}_{n+1}=(m_{\beta}v_{i_{0}})(m_{\alpha}v_{n+1})=d\tilde{v}_{i_{0}}v_{n+1}$. Then, it is clear that
        $$Id\leq 2d\leq d\tilde{v}_{i_{0}}v_{n+1}<\left(\dfrac{n+2}{n+1}\right)\tilde{w}_{i_{0}}\tilde{w}_{n+1}.$$
        \item If $i_{0},j_{0}\in\{0,\dots,n\}$, then $\tilde{w}_{i_{0}}\tilde{w}_{j_{0}}=m_{\beta}^{2}\tilde{v}_{i_{0}}\tilde{v}_{j_{0}}$. By Lemmas \ref{lem:4.3.2} and \ref{lem:4.3.3}, we have $m_{\alpha}<\left(\frac{n+2}{n+1}\right)\tilde{v}_{i_{0}}\tilde{v}_{j_{0}}$. Then
        $$Id\leq 2d=2m_{\alpha}m_{\beta}<\left(\frac{n+2}{n+1}\right)2m_{\beta}\tilde{v}_{i_{0}}\tilde{v}_{j_{0}}\leq\left(\frac{n+2}{n+1}\right)\tilde{w}_{i_{0}}\tilde{w}_{j_{0}}.$$
        \item If $i_{0}=n+1$ and $j_{0}=n+2$, we have $\tilde{w}_{n+1}\tilde{w}_{n+2}=w_{n+1}w_{n+2}$. Since $\mathbf{w}$ and $d$ verify equation (\ref{eq:3.1}), we conclude that $Id<\left(\frac{n+2}{n+1}\right)\tilde{w}_{n+1}\tilde{w}_{n+2}$.
    \end{itemize}
\end{proof}

As consequence  of Theorem 7.1 we have the following theorem.

\begin{theo}\label{cor:4.4.2}
    Let $f$ be a polynomial of type I of degree $d=m_{\alpha}m_{\beta}$ and weight vector $\mathbf{w}=(m_{\beta}\mathbf{v},m_{\alpha}v_{n+1},m_{\alpha}v_{n+2})$ which verify the conditions given in (\ref{eq:4.5}), such that $I=|\mathbf{w}|-d=1$. We denote its dual polynomial as $f^{T}$. Then the links  $L_{f}$ and $L_{f^{T}}$ are rational homology spheres. Moreover, if $m_{\alpha}>m_{\beta}$, then $L_{f}$ and $L_{f^T}$ admit a Sasaki-Einstein metric.
\end{theo}

\begin{proof}
    The first part was proven in the previous theorem. Now, we assume that $m_{\alpha}>m_{\beta}$. By Lemma \ref{lem:4.1.3}, we have $v_{n+1}=v_{n+2}=1$. We write $w_{i_{0}}w_{j_{0}}=\min_{i,j}\{w_{i}w_{j}\}$. As in Theorem \ref{prop:4.4.1}, when $i_{0},j_{0}\in\{0,\dots,n\}$ or when $i_{0}\in\{0,\dots,n\}$ and $j_{0}\in\{n+1,n+2\}$, we see that the weight vector $\mathbf{w}$ and the degree $d$ verify inequality (\ref{eq:3.1}). In the other case, when $i_{0}=n+1$ and $j_{0}=n+2$, we have $w_{n+1}=w_{n+2}=m_{\alpha}$. As $m_{\alpha}>m_{\beta}$, we obtain
    $$d=m_{\alpha}m_{\beta}<\left(\dfrac{n+2}{n+1}\right)m_{\alpha}^{2}=\left(\dfrac{n+2}{n+1}\right)w_{n+1}w_{n+2}=\left(\dfrac{n+2}{n+1}\right)\min_{i,j}\{w_{i}w_{j}\}.$$
    Thus, $L_{f}$ admits a Sasaki-Einstein metric. By Theorem \ref{prop:4.4.1}, it follows that $L_{f^{T}}$ also admits a Sasaki-Einstein metric.
\end{proof}

\begin{rem} \label{rem:4.4.3}There exist polynomials of type III whose weight vector has the form $\mathbf{w}=(m_{\beta}\mathbf{v},m_{\alpha},m_{\alpha})$. In this case, Theorem \ref{cor:4.4.2} also holds. 
\end{rem}

Next, we discuss what happens when $f$ is a polynomial of type II.

\begin{theo}\label{prop:4.4.4}
     Let $$f=f_{0}+z_{n+1}^{a_{n+1}}+z_{n+1}z_{n+2}^{a_{n+2}}$$ be a polynomial of type II of degree $d=m_{\alpha}m_{\beta}$ and weight vector $\mathbf{w}=(m_{\beta}\mathbf{v},m_{\alpha}v_{n+1},m_{\alpha}v_{n+2})$ which verify the conditions given in (\ref{eq:4.5}) and $\gcd(m_{\beta},a_{n+2}-1)=1$  such that $I=|\mathbf{w}|-d=1$ or $I=2$. We denote by $f^T$ the polynomial obtained by the B-H rule.  Then the links  $L_{f}$ and $L_{f^{T}}$ are rational homology spheres. Moreover, we have
     \begin{itemize}
         \item[(i)] If $I=1$ and $2\leq a_{n+2}<m_{\alpha}$ then $L_{f^T}$ admits a Sasaki-Einstein metric.
         \item[(ii)] If $I=2$ and $3\leq a_{n+2}<\frac{m_{\alpha}}{2}$ then $L_{f^T}$ admits a Sasaki-Einstein metric.
     \end{itemize}     
\end{theo}

\begin{proof}
     In a similar way as in Theorem \ref{prop:4.4.1}, the fact that $I=|\mathbf{w}|-d=1$ implies that $f_{0}$ and $f_{0}^{T}$ verify the condition $K1$. Then, by Theorem \ref{prop:4.2.6} we have that $L_{f}$ and $L_{f^{T}}$ are rational homology spheres.
     
     On the other hand, to verify that $L_{f^{T}}$ admits a Sasaki-Einstein metric in any case, we can assume without loss of generality that the dual polynomial $f^{T}=f_{0}^{T}+z_{n+2}z_{n+1}^{a_{n+1}}+z_{n+2}^{a_{n+2}}$ has degree $\tilde{d}=a_{n+2}m_{\alpha}m_{\beta}$ and weight vector 
     $$\mathbf{\tilde{w}}=(\tilde{w}_{0},\dots,\tilde{w}_{n+2})=(a_{n+2}m_{\beta}\tilde{v}_{0},\dots,a_{n+2}m_{\beta}\tilde{v}_{n},m_{\alpha}(a_{n+2}-1),m_{\alpha}m_{\beta}).$$
     This is possible because if $\tilde{d}$ and $\tilde{\mathbf{w}}$ verify inequality (\ref{eq:3.1}), then any multiple of these verify it as well.
     
     Now, we suppose that $I=1$ and $2\leq a_{n+2}<m_{\alpha}$. Let us see that $L_{f^{T}}$ admits a Sasaki-Einstein metric. For this, we must compute the index $\tilde{I}=|\mathbf{\tilde{w}}|-\tilde{d}$. Since $|\mathbf{\tilde{v}}|=|\mathbf{v}|$, we have
     \begin{align*}
         \tilde{I} & =a_{n+2}m_{\beta}|\mathbf{\tilde{v}}|+m_{\alpha}(a_{n+2}-1)+m_{\alpha}m_{\beta}-a_{n+2}m_{\alpha}m_{\beta}\\
         & =a_{n+2}m_{\beta}|\mathbf{v}|+m_{\alpha}(a_{n+2}-1)+m_{\alpha}m_{\beta}-a_{n+2}m_{\alpha}m_{\beta}\\
         & =a_{n+2}(m_{\beta}|\mathbf{v}|-m_{\alpha}m_{\beta})+m_{\alpha}(a_{n+2}-1+m_{\beta}).
     \end{align*}
     As $I=m_{\beta}|\mathbf{v}|+m_{\alpha}(v_{n+1}+v_{n+2})-m_{\alpha}m_{\beta}=1$, we obtain
     \begin{equation}\label{eq:4.26}
         \tilde{I}=a_{n+2}(1-m_{\alpha}(v_{n+1}+v_{n+2}))+m_{\alpha}(a_{n+2}-1+m_{\beta}).
     \end{equation}
     By Lemma \ref{lem:4.1.3}, we know $v_{n+1}=1$ and $a_{n+1}=m_{\beta}$. Since $w_{n+1}+a_{n+2}w_{n+2}=d$, we obtain $m_{\beta}=1+a_{n+2}v_{n+2}$. Replacing in (\ref{eq:4.26}), we obtain
     $$\tilde{I}=a_{n+2}(1-m_{\alpha}(1+v_{n+2}))+m_{\alpha}(a_{n+2}+a_{n+2}v_{n+2})=a_{n+2}.$$
     Next, we prove that $\tilde{d}$, $\tilde{I}$ and $\mathbf{\tilde{w}}$ verify inequality (\ref{eq:3.1}). We write $\tilde{w}_{i_{0}}\tilde{w}_{j_{0}}=\min_{i,j}\{\tilde{w}_{i}\tilde{w}_{j}\}$. Here, we consider the following cases:
     \begin{itemize}
         \item If $i_{0},j_{0}\in\{0,\dots,n\}$, then $\tilde{w}_{i_{0}}\tilde{w}_{j_{0}}=a_{n+2}^{2}m_{\beta}^{2}\tilde{v}_{i_{0}}\tilde{v}_{j_{0}}$. From Lemmas \ref{lem:4.3.2} and \ref{lem:4.3.3}, we have $m_{\alpha}<\left(\frac{n+2}{n+1}\right)\tilde{v}_{i_{0}}\tilde{v}_{j_{0}}$. Thus, we obtain
         $$\tilde{I}\tilde{d}=a_{n+2}^{2}m_{\alpha}m_{\beta}<\left(\frac{n+2}{n+1}\right)a_{n+2}^{2}\tilde{v}_{i_{0}}\tilde{v}_{j_{0}}m_{\beta}<\left(\dfrac{n+2}{n+1}\right)\tilde{w}_{i_{0}}\tilde{w}_{j_{0}}=\left(\dfrac{n+2}{n+1}\right)\min_{i,j}\{\tilde{w}_{i}\tilde{w}_{j}\}.$$
         \item If $i_{0}\in\{0,\dots,n\}$ and $j_{0}=n+1$, then $\tilde{w}_{i_{0}}\tilde{w}_{n+1}=m_{\alpha}m_{\beta}\tilde{v}_{i_{0}}a_{n+2}(a_{n+2}-1)$. As $\tilde{I}\tilde{d}=a_{n+2}^{2}m_{\alpha}m_{\beta}$, $\tilde{v}_{i_{0}}> 2$ and $a_{n+2}\geq 2$, we obtain
         $$\tilde{I}\tilde{d}\leq 2a_{n+2}(a_{n+2}-1)m_{\alpha}m_{\beta}< \tilde{v}_{i_{0}}a_{n+2}(a_{n+2}-1)m_{\alpha}m_{\beta}<\left(\dfrac{n+2}{n+1}\right)\min_{i,j}\{\tilde{w}_{i}\tilde{w}_{j}\}.$$
         \item If $i_{0}\in\{0,\dots,n\}$ and $j_{0}=n+2$, we have $\tilde{w}_{i_{0}}\tilde{w}_{n+2}=a_{n+2}\tilde{v}_{i_{0}}m_{\alpha}m_{\beta}^{2}$. As $m_{\beta}=1+a_{n+2}v_{n+2}>a_{n+2}$, we obtain
         $$\tilde{I}\tilde{d}=a_{n+2}^{2}m_{\alpha}m_{\beta}<a_{n+2}m_{\alpha}m_{\beta}^{2}<\left(\dfrac{n+2}{n+1}\right)\tilde{w}_{i_{0}}\tilde{w}_{j_{0}}=\left(\dfrac{n+2}{n+1}\right)\min_{i,j}\{\tilde{w}_{i}\tilde{w}_{j}\}.$$
         \item If $i_{0}=n+1$ and $j_{0}=n+2$, we have $w_{i_{0}}w_{j_{0}}=m_{\alpha}^{2}m_{\beta}(a_{n+2}-1)$. Then
         $$\tilde{I}\tilde{d}<\left(\dfrac{n+2}{n+1}\right)\min_{i,j}\{\tilde{w}_{i}\tilde{w}_{j}\}\Longleftrightarrow a_{n+2}^{2}m_{\alpha}m_{\beta}<\left(\dfrac{n+2}{n+1}\right)m_{\alpha}^2m_{\beta}(a_{n+2}-1).$$
         Simplifying, this inequality is equivalent to
         \begin{equation}\label{eq:4.27}
             \dfrac{a_{n+2}^2}{a_{n+2}-1}<\left(\dfrac{n+2}{n+1}\right)m_{\alpha}.
         \end{equation}
         Here, we consider two situations:
         \begin{itemize}
             \item If $n+2\leq a_{n+2}< m_{\alpha}$, then $\frac{a_{n+2}}{a_{n+2}-1}\leq\frac{n+2}{n+1}$. This implies the inequality in (\ref{eq:4.27}).
             \item If $a_{n+2}\leq n+1$, then $a_{n+2}^2\leq (n+1)^2\leq2^{n+1}+1\leq 1+a_{0}\dots a_{n}=m_{\alpha},$ for $n\geq 2$. As $\frac{1}{a_{n+2}-1}<\frac{n+2}{n+1}$, then (\ref{eq:4.27}) holds.
         \end{itemize}
     \end{itemize}
     Therefore, the link $L_{f^{T}}$ admits a Sasaki-Einstein metric in this case.

     On the other hand, we suppose that $I=2$ and $3\leq a_{n+2}< m_{\alpha}$. In a similar way as above, we have $\tilde{I}=2a_{n+2}$. Next, we prove that $\tilde{d}$, $\tilde{I}$ and $\mathbf{\tilde{w}}$ verify inequality (\ref{eq:3.1}). We consider $\tilde{w}_{i_{0}}\tilde{w}_{j_{0}}=\min_{i,j}\{\tilde{w}_{i}\tilde{w}_{j}\}$. Let us see the following cases:
     \begin{itemize}
         \item If $i_{0},j_{0}\in\{0,\dots,n\}$, then $\tilde{w}_{i_{0}}\tilde{w}_{j_{0}}=a_{n+2}^{2}m_{\beta}^{2}\tilde{v}_{i_{0}}\tilde{v}_{j_{0}}$. From Lemmas \ref{lem:4.3.2} and \ref{lem:4.3.3}, we have $m_{\alpha}<\left(\frac{n+2}{n+1}\right)\tilde{v}_{i_{0}}\tilde{v}_{j_{0}}$. Thus, we obtain
         $$\tilde{I}\tilde{d}=2a_{n+2}^{2}m_{\alpha}m_{\beta}<2\left(\frac{n+2}{n+1}\right)a_{n+2}^{2}\tilde{v}_{i_{0}}\tilde{v}_{j_{0}}m_{\beta}\leq\left(\dfrac{n+2}{n+1}\right)\tilde{w}_{i_{0}}\tilde{w}_{j_{0}}.$$
         \item If $i_{0}\in\{0,\dots,n\}$ and $j_{0}=n+1$, then $\tilde{w}_{i_{0}}\tilde{w}_{n+1}=m_{\alpha}m_{\beta}\tilde{v}_{i_{0}}a_{n+2}(a_{n+2}-1)$. As $\tilde{I}\tilde{d}=2a_{n+2}^{2}m_{\alpha}m_{\beta}$, $\tilde{v}_{i_{0}}\geq3$ and $a_{n+2}\geq 3$, we obtain
         \begin{align*}
           \tilde{I}\tilde{d} =2a_{n+2}^{2}m_{\alpha}m_{\beta} & \leq 3a_{n+2}(a_{n+2}-1)m_{\alpha}m_{\beta}  \\
           & \leq \tilde{v}_{i_{0}}a_{n+2}(a_{n+2}-1)m_{\alpha}m_{\beta} \\
            & < \left(\dfrac{n+2}{n+1}\right)\min_{i,j}\{\tilde{w}_{i}\tilde{w}_{j}\}.
         \end{align*}
         \item If $i_{0}\in\{0,\dots,n\}$ and $j_{0}=n+2$, we have $\tilde{w}_{i_{0}}\tilde{w}_{n+2}=a_{n+2}\tilde{v}_{i_{0}}m_{\alpha}m_{\beta}^{2}$. As $m_{\beta}=1+a_{n+2}v_{n+2}>2a_{n+2}$, we obtain
         $$\tilde{I}\tilde{d}=2a_{n+2}^{2}m_{\alpha}m_{\beta}<a_{n+2}m_{\alpha}m_{\beta}^{2}<\left(\dfrac{n+2}{n+1}\right)\tilde{w}_{i_{0}}\tilde{w}_{n+2}=\left(\dfrac{n+2}{n+1}\right)\min_{i,j}\{\tilde{w}_{i}\tilde{w}_{j}\}.$$
         \item If $i_{0}=n+1$ and $j_{0}=n+2$, we have $w_{i_{0}}w_{j_{0}}=m_{\alpha}^{2}m_{\beta}(a_{n+2}-1)$. Then
         $$\tilde{I}\tilde{d}<\left(\dfrac{n+2}{n+1}\right)\min_{i,j}\{\tilde{w}_{i}\tilde{w}_{j}\}\Longleftrightarrow 2a_{n+2}^{2}m_{\alpha}m_{\beta}<\left(\dfrac{n+2}{n+1}\right)m_{\alpha}^2m_{\beta}(a_{n+2}-1).$$
         Simplifying, this inequality is equivalent to
         \begin{equation}\label{eq:4.28}
             \dfrac{a_{n+2}^2}{a_{n+2}-1}<\left(\dfrac{n+2}{n+1}\right)\dfrac{m_{\alpha}}{2}.
         \end{equation}
         Here, we consider two situations:
         \begin{itemize}
             \item If $(n+2)\leq a_{n+2}< \frac{m_{\alpha}}{2}$, then $\frac{a_{n+2}}{a_{n+2}-1}\leq\frac{n+2}{n+1}$. This implies the inequality in (\ref{eq:4.28}).
             \item If $a_{n+2}\leq n+1$, then $a_{n+2}^2\leq (n+1)^2\leq2^{n+1}+1\leq 1+a_{0}\dots a_{n}=m_{\alpha}$, for $n\geq 2$. As $\frac{2}{a_{n+2}-1}<\frac{n+2}{n+1}$, then (\ref{eq:4.28}) holds.
         \end{itemize}
     \end{itemize}
\end{proof}

\begin{rem}\label{re:7.2.2} In case the Fano index equal 2 and  $\gcd(m_{\beta},a_{n+2}-1)>1$  Theorem \ref{prop:4.4.4} is still valid provided 
$\gcd(m_{\alpha},a_{n+2})=1$ according to Corollary \ref{cor:5.4.4} 
\end{rem}

Next, we exhibit some examples.

\begin{exm}\label{ex:4.4.5}
    We consider the weight vector $\mathbf{w}=(161,28,147,67,67)$ of the sporadic list of Johnson and Kollár in \cite{JK}. For this weight $\mathbf{w}$, we can associate three invertible polynomials (type I, II and III, respectively):
    \begin{align*}
        f_{I} & =z_{2}z_{0}^{2}+z_{0}z_{1}^{11}+z_{1}z_{2}^{3}+z_{3}^{7}+z_{4}^{7}, \\
        f_{II} & = z_{2}z_{0}^{2}+z_{0}z_{1}^{11}+z_{1}z_{2}^{3}+z_{3}^{7}+z_{3}z_{4}^{6}, \\    f_{III} & = z_{2}z_{0}^{2}+z_{0}z_{1}^{11}+z_{1}z_{2}^{3}+z_{4}z_{3}^{6}+z_{3}z_{4}^{6},
    \end{align*}
    where degree $d=m_{\alpha}m_{\beta}$, $m_{\alpha}=67$ and $m_{\beta}=7$. By Theorem \ref{prop:4.1.4}, we have that the links $L_{f_{I}}$, $L_{f_{II}}$ and $L_{f_{III}}$ are rational homology $7$-spheres and
    $$H_{3}(L_{f_{I}},\mathbb{Z})=H_{3}(L_{f_{II}},\mathbb{Z})=H_{3}(L_{f_{III}},\mathbb{Z})=(\mathbb{Z}_{67})^{6}.$$
    Now, we consider their dual polynomials
    \begin{align*}
        f_{I}^{T} & =z_{0}^{2}z_{1}+z_{1}^{11}z_{2}+z_{2}^{3}z_{0}+z_{3}^{7}+z_{3}^{7},\\
        f_{II}^{T} & =z_{0}^{2}z_{1}+z_{1}^{11}z_{2}+z_{2}^{3}z_{0}+z_{3}^{7}z_{4}+z_{4}^{6},\\
        f_{III}^{T} & =z_{0}^{2}z_{1}+z_{1}^{11}z_{2}+z_{2}^{3}z_{0}+z_{4}z_{3}^{6}+z_{3}z_{4}^{6}.
    \end{align*}
    As the index $I=|\mathbf{w}|-d=1$, by  Theorems \ref{prop:4.2.4} and \ref{prop:4.2.6}, we have that the links $L_{f_{I}^{T}}$, $L_{f_{II}^{T}}$ and $L_{f_{III}^{T}}$ are rational homology $7$-spheres where
    $$H_{3}(L_{f_{I}^{T}},\mathbb{Z})=H_{3}(L_{f_{III}^{T}},\mathbb{Z})=(\mathbb{Z}_{67})^{6} \quad \text{ and } \quad H_{3}(L_{f_{II}^{T}},\mathbb{Z})=\mathbb{Z}_{67}$$
    Moreover, as the weight vector $\mathbf{w}$ and the degree $d$ satisfy the inequality (\ref{eq:3.1}), by Theorem \ref{prop:4.4.1} we have that the links $L_{f_{I}^{T}}$ and $L_{f_{III}^{T}}$ admit a Sasaki-Einstein structure. In addition, since $2\leq a_{n+2}=6<m_{\alpha}=67$, we have by Theorem \ref{prop:4.4.4} that the link $L_{f_{II}^{T}}$ also admits a Sasaki-Einstein metric.
\end{exm}

\begin{exm} Let us  consider the polynomial $f$ of type I:
$$f=z_{4}z_{0}^{2}+z_{0}z_{1}^{3}+z_{1}z_{2}^{11}+z_{2}z_{3}^{12}+z_{3}z_{4}^{8}+z_{5}^{19}+z_{6}^{19}$$
whose degree is $d=m_{\alpha}m_{\beta}=120403$, where $m_{\alpha}=6337$ and $m_{\beta}=19$, and its weight vector 
$$\mathbf{w}=(53257,22382,8911, 9291, 13889, 6337, 6337).$$
Clearly, we notice that the index $I=|\mathbf{w}|-d=1$, $\gcd(m_{\alpha},m_{\beta})=1$ and $m_{\alpha}>m_{\beta}$. Then, by Theorem \ref{cor:4.4.2}, we have that the links $L_{f}$ and $L_{f^{T}}$ are rational homology $11$-spheres and admit a Sasaki-Einstein metric, where $f^{T}$ is the dual polynomial of $f$ obtained by the B-H rule:
$$f^{T}= z_{0}^{2}z_{1}+z_{1}^{3}z_{2}+z_{2}^{11}z_{3}+z_{3}^{12}z_{4}+z_{4}^{8}z_{0}+z_{5}^{19}+z_{6}^{19}.$$
\end{exm}

\begin{exm}\label{ex:4.4.7} Let us consider the following polynomial of type $II$:
$$f=z_{4}z_{0}^{2}+z_{0}z_{1}^{3}+z_{1}z_{2}^{11}+z_{2}z_{3}^{6}+z_{3}z_{4}^{16}+z_{5}^{29}+z_{5}z_{6}^{28}$$ which admits 
 the weight vector $\mathbf{w}=(87029,32248,13775,28333,9715, 6337, 6337).$ 
 It follows that the  corresponding weighted hypersurface $X_f$ has degree $d=183773$ and has Fano index $I=|\mathbf{w}|-d=1.$ Thus  
 $\gcd(a_{n+2}-1, m_{\beta})=1$ (here $m_{\alpha }=6337$ and $m_{\beta}=29$  and $a_{n+2}=28).$ 
 The B-H transpose of $f$ is given by 
 $$f^{T}= z_{0}^{2}z_{1}+z_{1}^{3}z_{2}+z_{2}^{11}z_{3}+z_{3}^{6}z_{4}+z_{4}^{16}z_{0}+z_{5}^{29}z_{6}+z_{6}^{28}.$$ 
 It follows from  Proposition \ref{prop:4.1.4} and Theorem \ref{prop:4.2.6} that $L_f$ and its  B-H dual link $L_{f^T}$ are  rational homology 11-sphere. Moreover from Theorem \ref{prop:4.4.4}  the link $L_{f^{T}}$ admits a Sasaki-Einstein metric.
\end{exm} 

\begin{exm}\label{ex:4.4.6}
    In \cite{BK}, Brown and Kasprzyk found 7084 sporadic cases of well-formed quasismooth threefold hypersuface of index $I=2$. From these, we consider the weight vector $\mathbf{w}=(128,56,40,37,37)$ and the hypersurface $X_{f}\subset\mathbb{P}(\mathbf{w})$ defined by the polynomial
    $$f=z_{2}z_{0}^{2}+z_{0}z_{1}^{3}+z_{1}z_{2}^{6}+z_{3}^{8}+z_{3}z_{4}^{7}$$
    of degree $d=m_{\alpha}m_{\beta}=296$, where $m_{\alpha}=37$ and $m_{\beta}=8$. Clearly, we see that $I=|\mathbf{w}|-d=2$. Then, applying the B-H rule we obtain the dual polynomial $$f^{T}=z_{0}^{2}z_{1}+z_{1}^{3}z_{2}+z_{2}^{6}z_{0}+z_{3}^{8}z_{4}+z_{4}^{7}.$$
    By Proposition \ref{prop:4.1.4} and  Corollary \ref{cor:5.4.4}) we have that the links $L_{f}$ and $L_{f^{T}}$ are rational homology $7$-spheres. Moreover, as $3\leq a_{4}=7<\frac{m_{\alpha}}{2}=\frac{37}{2}$, then by   Remark \ref{re:7.2.2}) we have that the link $L_{f^{T}}$
    admits a Sasaki-Einstein metric.
\end{exm}